\UseRawInputEncoding
\documentclass[11pt]{amsart}
\usepackage{tabularx,booktabs}
\usepackage{caption}
\usepackage{amsmath}
\usepackage{amsfonts}

\usepackage{amscd}
\usepackage{amsthm}
\usepackage{amssymb} \usepackage{latexsym}
\usepackage{eufrak}
\usepackage{euscript}
\usepackage{epsfig}
\usepackage{graphics}
\usepackage{array}
\usepackage{enumerate}
\usepackage{dsfont}
\usepackage{color}
\usepackage{wasysym}
\usepackage{hyperref}
\usepackage{pdfsync}
\usepackage{soul}

\newcommand{\bel}[1]{\begin{equation}\label{#1}}

\newcommand{\be}{\begin{equation}}

\newcommand{\ba}{\begin{eqnarray}}
\newcommand{\ea}{\end{eqnarray}}

\newcommand{\qe}{\end{equation}}

\newcommand{\Hmm}[1]{\leavevmode{\marginpar{\tiny%
$\hbox to 0mm{\hspace*{-0.5mm}$\leftarrow$\hss}%
\vcenter{\vrule depth 0.1mm height 0.1mm width \the\marginparwidth}%
\hbox to
0mm{\hss$\rightarrow$\hspace*{-0.5mm}}$\\\relax\raggedright #1}}}

\newtheorem{theorem}{Theorem}[section]

\newtheorem{lemma}[theorem]{Lemma}
\newtheorem{corollary}[theorem]{Corollary}
\newtheorem{definition}[theorem]{Definition}

\newtheorem{remark}[theorem]{Remark}

\newtheorem{prop}[theorem]{Proposition}

\makeatletter

\@addtoreset{equation}{section}

\makeatother
\begin{document}

\title[on subelliptic harmonic maps with potential]{on subelliptic harmonic maps with potential}

\author{Yuxin Dong}\footnote{The first author is supported by NSFC Grants No. 11771087 and No. 12171091, and the second author is the
corresponding author.}
\address{Yuxin Dong: School of Mathematical Sciences, Laboratory of Mathematics for Nonlinear Science, Fudan University, Shanghai 200433, P.R. China }
\email{\href{mailto:yxdong@fudan.edu.cn}{yxdong@fudan.edu.cn}}

\author{Han Luo}
\address{Han Luo: School of Mathematical Sciences, Fudan University, Shanghai 200433, P.R. China }
\email{\href{mailto:19110180023@fudan.edu.cn}{19110180023@fudan.edu.cn}}

\author{Weike Yu}
\address{Weike Yu: School of mathematics and statistics, Nanjing University of Science and Technology, Nanjing, 210094, Jiangsu, P.R. China }
\email{\href{mailto:wkyu2018@outlook.com}{wkyu2018@outlook.com}}

\begin{abstract}
Let $(M,H,g_H;g)$ be a sub-Riemannian manifold and $(N,h)$ be a Riemannian manifold. For a smooth map $u: M \to N$, we consider the energy functional $E_G(u) = \frac{1}{2} \int_M[|\mathrm{d}u_H|^2 - 2G(u)] \mathrm{d}V_M$, where $\mathrm{d}u_H$ is the horizontal differential of $u$, $G:N\to \mathbb{R}$ is a smooth function on $N$. The critical maps of $E_G(u)$ are referred to as subelliptic harmonic maps with potential $G$. In this paper, we investigate the existence problem for subelliptic harmonic maps with potentials by a subelliptic heat flow. Assuming that the target Riemannian manifold has non-positive sectional curvature and the potential $G$ satisfies various suitable conditions, we prove some Eells-Sampson type existence results when the source manifold is either a step-$2$ sub-Riemannian manifold or a step-$r$ sub-Riemannian manifold whose sub-Riemannian structure comes from a tense Riemannian foliation.
\end{abstract}

\maketitle

\section{Introduction}\label{sec:intro}

Sub-Riemannian geometry can be regraded as a natural generalization of Riemannian geometry. A sub-Riemannian manifold is defined as a triple $(M, H, g_H)$, where $M$ is a connected smooth manifold, $H$ is a subbundle bracket generating for $TM$, and $g_H$ is a smooth fiberwise metric on $H$. Recently, geometric analysis on sub-Riemannian manifolds has been the subject of intense study (cf.\cite{MR2154760}\cite{baudoin2018sublaplacians}).

On the other hand, the equilibrium system of ferromagnetic spin chain and the Neumann Motion which describe two important physic phenomena, have been received much attention. To study them, harmonic maps with potential were introduced in \cite{MR1433176}. This is a new kind of maps more general than the usual harmonic maps, whose behavior may drastically change in the presence of a potential. Various existence and nonexistence results of harmonic
maps with potential have been achieved. Fardoun and Ratto \cite{MR1433176} obtained some variational properties and existence results of harmonic maps with potential when the target manifolds are spheres. Chen \cite{MR1618210} gave a Liouville theorem for harmonic maps with potential and he \cite{MR1680678} also established uniqueness and
existence results of Landau-Lifshitz equations. Later, \cite{MR1800592} gave Eells-Sampson type results for harmonic maps with potential.

The main purpose of this paper is to study a natural counterpart of harmonic maps with potential in sub-Riemannian geometry. Let $(M, H, g_H)$ be a sub-Riemannian manifold with a smooth measure $d\mu$ and let $(N, h)$ be a Riemannian manifold. Given a function $G \in C^{\infty}(N)$, we consider the following energy functional
\begin{equation}\label{eq:energy functional}
E_G(u)=\frac{1}{2}\int_M [\vert \mathrm{d}u_H\vert^2-2G(u)]\,\mathrm{d}\mu
\end{equation}
where $u: M \to N$ is a smooth map and $\mathrm{d}u_H$ is the restriction of $\mathrm{d}u$ to $H$. A smooth map $u$ : $(M, H, g_H)\to (N, h)$ is called a subelliptic harmonic map with potential $G$ if it is a critical point of (\ref{eq:energy functional}). The subelliptic harmonic maps with potential can be viewed as a generalization of both harmonic maps with potential and subellipitic harmonic maps. In sub-Riemannian geometry, some most important cases are the relatively simple cases that the measures $\mathrm{d}\mu$ for sub-Riemannian manifolds are the volume elements of compatible Riemannian metrics (also called Riemannian extension $g$ of $g_H$), such as the Webster metric in CR geometry, the contact metric in contact geometry, etc. Hence, the present paper concentrates on sub-Riemannian manifolds with such Riemannian extensions. Then the Euler-Lagrange equation of (\ref{eq:energy functional}) is
\begin{equation}\label{eq:tension field}
\tau(u)=\tau_H(u)+(\widetilde\nabla G)(u)=0
\end{equation}
where $\tau_H(u)$ denotes the subelliptic tension field associated with the horizontal energy (cf.\cite{MR4236537}) and $\widetilde\nabla$ is the Riemannian connection on $(N, h)$.

As we know, \cite{10.2307/117803} first introduced subelliptic harmonic maps whose domain are in the Euclidean space. They investigated the Dirichlet existence problem for such subelliptic harmonic maps. Later, a related uniqueness result was given by \cite{226991494}. The pseudo-harmonic maps from pseudoconvex CR manifolds, introduced by Barletta et al. in \cite{MR1871387}, are actually subelliptic harmonic maps defined with respect to the Webster metrics. Some regularity results for subelliptic harmonic maps from Carnot groups were established in \cite{226227263}, see also \cite{225602822}, \cite{272379899} for some regularity results of subelliptic p-harmonic maps. In \cite{barletta2004pseudoharmonic}, the authors investigated the stability problem of pseudo-harmonic maps with potential from strictly pseudoconvex CR manifolds, which are the special cases of the subelliptic harmonic maps with potential.

In Riemannian and K{\"a}hlerian geometry, there are many fundamental applications based on Eells-Sampson theorem (cf.\cite{10.2307/117803}\cite{2347703}). Chang and Chang \cite{MR3038722} obtained an Eells-Sampson type result for pseudo-harmonic maps from CR manifolds under some additional analytic condition, which was later generalized by Ren and Yang in \cite{MR3844509}. Dong \cite{MR4236537} obtained Eells-Sampson type results for subelliptic harmonic maps from sub-Riemannian manifolds in some general cases. In this paper, we aim to establish Eells-Sampson type theorems for subelliptic harmonic maps with potential from certain kinds of sub-Riemannian manifolds. To this end, we investigate the following subelliptic harmonic map heat flow with potential $G$ associated with (\ref{eq:energy functional}) and (\ref{eq:tension field}):
\begin{equation}\label{eq:heat flow u}
\left\{\begin{gathered}
\frac{\partial u}{\partial t}=\tau(u)\\
\left.u\right|_{t=0}=\bar u
\end{gathered}\right.
\end{equation}
where $\bar u: M \to N$ is the initial map, which is assumed to be $C^{\infty}$ for simplicity.

Here and afterwards, let $\text{Hess}\ G$ denote the Hessian matrix of $G$ with
respect to the Riemannian connection $\widetilde\nabla$. First we have a long time existence for (\ref{eq:heat flow u}) as follows.

\begin{theorem}\label{thm:long time existence for complete}
Let $(M,H,g_H;g)$ be a compact sub-Riemannian manifold and let $(N,h)$ be a complete Riemannian mainfold with non-positive sectional curvature. There exists $T>0$ such that the heat flow (\ref{eq:heat flow u}) has a unique solution on $M\times[0,T)$. If, for some constant $C > 0$, $\text{Hess}\ G\le C\cdot h$, then $T = + \infty$. In particular, when $(N,h)$ is compact, for any $G\in C^{\infty}(N)$, we have $T = + \infty$.
\end{theorem}

For a more precise statement, a generating order was introduced for the sub-Riemannian manifold. We say $M$ is a step-$r$ sub-Riemannian manifold if the tangent space $T_xM$ at each point $x$ can be spanned by sections of $H$ together with their Lie brackets up to order $r$. The simpliest nontrivial sub-Riemannian manifolds are step-$2$ sub-Riemannian manifolds, which include strictly pseudoconvex CR manifolds, contact metric manifolds, and quaternionic contact manifolds, etc. On the other hand, Riemannian foliations provide an important source of sub-Riemannian manifolds as well. For a Riemannian foliation $(M, g; \mathfrak{F})$ with a bundle-like metric $g$, let $H = (T\mathfrak{F})^{\perp}$ (the horizontal subbundle of the foliation $\mathfrak{F}$ with respect to $g$) and $g_H=g\vert_H$. If $H$ is bracket generating for $TM$, then we have a sub-Riemannian manifold $(M, H, g_H; g)$ corresponding to $(M, g; \mathfrak{F})$. The Riemannian foliation $(M, g; \mathfrak{F})$ will be said to be tense if the mean vector field of $\mathfrak{F}$ is parallel with respect to the Bott connection along the leaves.

In order to establish the Eells-Sampson type results, we need to establish the convergence of $u(\cdot,t)$ as $t\to\infty$. Now we consider the following two cases: the source manifold $(M,H,g_H;g)$ is a compact step-$2$ sub-Riemannian manifold or a compact sub-Riemannian manifold corresponding to
a tense Riemannian foliation with the property that $H$ is bracket generating for $TM$. Let: $S(V)\to M$ be the unit sphere bundle of the vertical bundle $V$, that is, $S(V)=\{v\in V :\Vert v \Vert_g=1\}$. For any $v\in S(V)$, the $v$-component of $T(\cdot,\cdot)$ is given by $T^v(\cdot,\cdot)=\langle T(\cdot,\cdot),v\rangle$, where $T(\cdot,\cdot)$ is the torsion of the generalized Bott connection $\nabla^{\mathfrak{B}}$. In the first case, it turns out that the smooth function $\eta(v)=\frac 1 2\Vert T^v(\cdot,\cdot)\Vert^2_g$ can achieve a positive minimal value on $S(V)$, that is, $\eta_{min}=min_{v\in S(V)}\eta(v)>0$.

\begin{theorem}\label{thm:eells sampson type for compact}
Let $(M,H,g_H;g)$ be a compact step-$2$ sub-Riemannian manifold and let $(N,h)$ be a compact Riemannian manifold with non-positive sectional curvature. If $\text{Hess}\,G<\frac{\eta_{min}}{2}\cdot h$, then, for any smooth map $\bar u:M\to N$, the subelliptic harmonic map heat flow with potential $G$ exists on $M\times[0,+\infty)$ and there exists a sequence $t_i\to \infty$, such that $u(x,t_i)$ converges uniformly to a subelliptic
harmonic map $u_{\infty}(x)$, as $t_i\to \infty$. In particular, any map $\bar u \in C^{\infty}(M,N)$ is homotopic to a $C^{\infty}$ subelliptic harmonic map with potential $G$.
\end{theorem}

When $N$ is complete, we need a decay condition on $G$ to obtain Eells-Sampson type results.
\begin{prop}\label{thm:hess assumption}
Suppose $(M,H,g_H;g)$ is a compact step-$2$ sub-Riemannian manifold. Let $(N,h)$ be a complete Riemannian manifold with non-positive sectional curvature and let $\rho$ denote the distance function on $N$ from a fixed point $P_{0}\in N$. If, for some $C> 0$,
\begin{equation}\label{eq:hess assumption}
\text{Hess}\,G(y)\le -C(1+\rho(y))^{-1}\cdot h
\end{equation}
then $T=+\infty$ and $u$ converges to $u_{\infty}$, where $u_{\infty}$ is a constant. Moreover, $\Sigma_{G}\neq \varnothing$ in this case, where $\Sigma_{G} = \{y \in N : \widetilde\nabla G(y) = 0\}$.
\end{prop}

For the second case that $M$ is a compact sub-Riemannian manifold corresponding to
a tense Riemannian foliation, we have
\begin{prop}\label{thm:hess assumption foliation}
Suppose $(M,H,g_H;g)$ is a compact sub-Riemannian manifold corresponding to
a tense Riemannian foliation with the property that $H$ is bracket generating for $TM$. Let $(N,h)$ be a complete Riemannian manifold with non-positive sectional curvature and let $\rho$ denote the distance function on $N$ from a fixed point $P_{0}\in N$. If, for some $C> 0$,
\begin{equation*}
\text{Hess}\,G(y)\le -C(1+\rho(y))^{-1}\cdot h
\end{equation*}
then $T=+\infty$ and $u$ converges to $u_{\infty}$, where $u_{\infty}$ is a constant. Moreover, $\Sigma_{G}\neq \varnothing$ in this case.
\end{prop}

Suppose now that $N$ is embedded isometrically into some Euclidean space $\mathbb{R}^{K}$. Let $\mathfrak J: N \to \mathbb{R}^{K}$ denote this embedding and $A(y): T_yN\times T_yN \to (T_yN)^{\perp}$ denote its second fundamental form. We also assume that the potential $G$ is the restriction to $N$ of some smooth function $\bar G: \mathbb{R}^{K} \to \mathbb{R}$. For simplicity, we write $y$ instead of $\mathfrak J(y)$ for all $y \in N$. Then Equation (\ref{eq:heat flow u}) becomes

\begin{equation}\label{eq:euclidean heat flow}
\left\{\begin{gathered}
\Delta_H u-\frac{\partial u}{\partial t}=A(u)(\mathrm{d}u_H,\mathrm{d}u_H)-P(D\,\bar G(u)) \\
u(x,0)=\bar u(x)
\end{gathered}\right.
\end{equation}
where $P:\mathbb{R}^{K} \to T_yN$ is the orthogonal projection onto the tangent space of $N$ at $y$, $D$ is the canonical Riemannian connection of $\mathbb{R}^{K}$.

\begin{prop}\label{thm:product assumption}
Let $(M,H,g_H;g)$ be either a compact step-$2$ sub-Rieman-\\nian manifold or a compact sub-Riemannian manifold corresponding to
a tense Riemannian foliation with the property that $H$ is bracket generating for $TM$. Let $(N,h)$ be a complete Riemannian manifold with non-positive sectional curvature and let $\mathfrak J: N \to \mathbb{R}^{K}$ be an isometric embedding. Suppose for all $y\in N, Y\in T_{y} N$,
\begin{equation}\label{eq:product assumption}
\langle A(y)(Y,Y),y\rangle_{\mathbb R^{K}}+\vert Y\vert^2_{\mathbb R^{K}}-\langle \widetilde\nabla G(y),y\rangle_{\mathbb R^{K}} \ge 0.
\end{equation}
In the first case, if $\text{Hess}\,G<\frac{\eta_{min}}{2}\cdot h$,
then $T=+\infty$ and $u$ subconverges to $u_{\infty}$, where $u_{\infty}$ is a smooth subelliptic harmonic map with potential $G$. In the second case, if $\text{Hess}\,G\le 0$, then $T=+\infty$ and $u$ subconverges to $u_{\infty}$, where $u_{\infty}$ is a constant. Moreover, $\Sigma_{G}\neq \varnothing$ in this case.
\end{prop}

\section{Preliminaries}

First, we recall some facts on sub-Riemannian manifolds. Suppose that $M$ is a connected smooth $(m+d)$--dimensional manifold and $H$ is a rank $m$ subbundle of tangent bundle $TM$. For any $x\in M$ and any open neighborhood $U$ of $x$, we denote the space of smooth sections of $H$ on $U$ by $\Gamma (U, H)$ and define $\{\Gamma^j (U, H)\}_{j\ge 1}$ inductively
by
\begin{equation*}
\begin{split}
\Gamma^1 (U, H) &= \Gamma (U, H) \\
\Gamma^{j+1} (U, H) &= \Gamma^j (U, H)+[\Gamma^1 (U, H), \Gamma^j (U, H)]
\end{split}
\end{equation*}
for each positive integer $j$, where $[\cdot,\cdot]$ denotes the Lie bracket of vector fields. At each point $x$, we
obtain a subspace $H^{(j)}_x$ of the tangent space $T_x M$, that is,
\begin{equation*}
H^{(j)}_x (U, H)=\{ X(x):X\in \Gamma^j (U, H)\}.
\end{equation*}
The subbundle $H$ is said to be $r$-step bracket generating for $TM$, if $H^{(r)}_x = T_xM$ for each $x \in M$ (cf.\cite{MR862049}, \cite{MR1867362}). In this paper, we always assume that $H$ satisfies the $r$-step bracket generating condition for some $r\ge 2$.

A sub-Riemannian manifold is a triple $(M,H,g_H)$, where $g_H$ is a fiberwise metric on $H$. According to \cite{MR862049}, there always exists a Riemannian metric $g$ on $M$ such that $g\vert_H=g_H$, where $g$ is referred to as a Riemannian extension of $g_H$. Henceforth, we always fix a Riemannian extension $g$ on the sub-Riemannian manifold $(M,H,g_H)$, and consider the quadruple $(M,H,g_H,g)$. In terms of the metric $g$, the tangent bundle $TM$ has the following orthogonal decomposition:
\begin{equation}\label{eq:orthogonal decomposition}
TM=H\oplus V
\end{equation}
which induces the projections $\pi_H : TM \to H$ and $\pi_V : TM \to V$. The distribution $H$ and $V=H^{\perp}$ are called the horizontal distribution and the vertical distribution respectively on $(M,H,g_H,g)$.

We consider the generalized Bott connection on sub-Riemannian manifolds, which is given by (cf.\cite{MR3587668},\cite{baudoin2015logsobolev},\cite{MR4236537})
\begin{equation}\label{eq:bott connection}
\nabla^{\mathfrak B}_X Y=\begin{cases}
\pi_H(\nabla^R_X Y),\qquad &X,Y\in\Gamma(H) \\
\pi_H([X,Y]),\qquad &X\in\Gamma(V),Y\in\Gamma(H) \\
\pi_V([X,Y]),\qquad &X\in\Gamma(H),Y\in\Gamma(V) \\
\pi_V(\nabla^R_X Y),\qquad &X,Y\in\Gamma(V)
\end{cases}
\end{equation}
where $\nabla^R$ denotes the Riemannian connection of $g$. It is convenient for computations on sub-Riemannian manifolds by using the above connection, since $\nabla^{\mathfrak B}$ preserves the decomposition (\ref{eq:orthogonal decomposition})
and it also satisfies
\begin{equation*}
\nabla^{\mathfrak B}_X g_H = 0 \quad \text{and} \quad \nabla^{\mathfrak B}_Y g_V = 0
\end{equation*}
for any $X \in H$ and $Y \in V$, where $g_V = g\vert_V$. However, in general, $\nabla^{\mathfrak{B}}$ is not compatible to $g$.

Let $(M^{m+d},H^m,g_H;g)$ be a step-$r$ sub-Riemannian manifold with a rank $m$ subbundle $H$. For a smooth function $u$ on $M$, its horizontal gradient is the unique vector field $\nabla^H u$ satisfying $(\nabla^H u)_q = \pi_H (\nabla^{\mathfrak{B}}u)_q$, for any $q\in M$. Let $\{e_i\}^{m}_{i=1}$ and $\{e_\alpha \}^{m+d}_{\alpha=m+1}$ be local orthonormal frame fields of the distributions $H$ and $V$ on an open domain $\Omega$ of $(M,g)$ respectively. We call the local orthonormal frame field $\{e_A\}^{m+d}_{A=1}$ an adapted frame field for $(M,H,g_H;g)$. As a result,
\begin{equation}\label{eq:horizontal gradient}
 \nabla^H u=\sum\limits^{m}_{i=1}(e_iu)e_i.
\end{equation}
Since $H$ has the bracket generating property for $TM$, $u$ is constant if and only if $\nabla^H u=0$.
The divergence of a vector field $X$ on $M$ is given by
\begin{equation*}
 div_gX=\sum\limits^{m+d}_{A=1}\lbrace e_A\langle X,e_A\rangle - \langle X,\nabla^R_{e_A}e_A\rangle\rbrace.
\end{equation*}
Then one can define sub-Laplacian of a function $u$ on $(M,H,g_H;g)$ as
\begin{equation}\label{eq:sub laplacian}
\Delta_H u=div_g(\nabla^H u).
\end{equation}
By the divergence theorem, it is clear that $\Delta_H$ is symmetric, that is,
\begin{equation*}
\int_M v(\Delta_H u)\,\mathrm{d}v_g=\int_M u(\Delta_H v)\,\mathrm{d}v_g=-\int_M \langle\nabla^H u,\nabla^H v\rangle\,\mathrm{d}v_g
\end{equation*}
for any $u,v\in C^{\infty}_{0}(M)$. In terms of (\ref{eq:horizontal gradient}) and (\ref{eq:sub laplacian}), we deduce that
\begin{equation}\label{eq:hormander type operator}
\begin{split}
\Delta_H u &=\sum\limits^m_{i=1}\lbrace e_i\langle\nabla^H u,e_i\rangle-\langle\nabla^H u,\nabla^{\mathfrak B}_{e_i}e_i\rangle\rbrace-\langle\nabla^H u,\zeta\rangle \\
&=\sum\limits^m_{i=1}e^2_i(u)-(\sum\limits^m_{i=1}\nabla^{\mathfrak B}_{e_i}e_i+\zeta)(u)
\end{split}
\end{equation}
where $\zeta=\pi_H(\sum\limits_{\alpha}\nabla^R_{e_{\alpha}}e_{\alpha})$ is referred to as the mean curvature vector field of the vertical distribution $V$.

Suppose $X_1, \cdots , X_m, Y$ are $C^{\infty}$ vector fields on a manifold $\widetilde M$, whose commutators up to certain order span the tangent space at each point. The so-called H{\"o}rmander operator
\begin{equation*}
\mathfrak{D}=\sum\limits^m_{i=1} X_i^2+Y
\end{equation*}
were first studied by H{\"o}rmander in \cite{MR222474}.
He proved a celebrated result that $\mathfrak{D}$
is hypoelliptic. In other words, if $u$ is a distribution defined on any open set $\Omega \subset \widetilde M$, such that $\mathfrak{D} u \in C^{\infty}(\Omega)$, then $u \in C^{\infty}(\Omega)$. In terms of (\ref{eq:hormander type operator}), we find that $\Delta_H$ is an operator of H{\"o}rmander type, and hence it is hypoelliptic on $M$. Furthermore, the operator $\Delta_H-\frac{\partial}{\partial t}$ is hypoelliptic as well.

Define
\begin{equation*}
S^p_k(\Delta_H,\Omega)=\{u\in L^p(\Omega)\vert e_{i_1}\cdots e_{i_s}(u)\in L^p(\Omega), 1\le i_1,\cdots,i_s\le m, 0\le s\le k\}
\end{equation*}
and
\begin{equation*}
\begin{split}
S^p_k(\Delta_H-\frac{\partial}{\partial t},\Omega\times [0,T))=\{&u\in L^p(\Omega\times [0,T))\vert \partial^l_t e_{i_1}\cdots e_{i_s}(u)\in L^p(\Omega\times [0,T)),\\
&1\le i_1,\cdots,i_s\le m, 2l+s\le k\}
\end{split}
\end{equation*}
for any nonnegative integer $k$. Due to regularity theory for hypoelliptic operators of Rothschild and Stein \cite{MR436223}, we have the following
\begin{theorem}\label{thm:subelliptic regularity theory}
Let $\mathfrak{D}=\Delta_{H}\,(resp.\ \Delta_{H}-\frac{\partial}{\partial t})$ and $\widetilde{M}=\Omega\,(resp.\ \Omega\times(0,T))$. Suppose $f\in L^p_{loc}(\widetilde{M})$, and
\begin{equation*}
\mathfrak{D}\,f=g\qquad on\quad \widetilde{M}.
\end{equation*}
If $g\in S^p_k(\mathfrak{D},\widetilde{M})$, then $\chi f\in S^p_{k+2}(\mathfrak{D},\widetilde{M})$ for any $\chi \in C^{\infty}_{0}(\widetilde{M})$, $1<p<\infty, k=0,1,2,\cdots$. In particular, the following inequality holds
\begin{equation*}
\Vert \chi f\Vert_{S^p_{k+2}(\mathfrak{D},\widetilde{M})}\le C_{\chi}(\Vert g\Vert_{S^p_{k}(\mathfrak{D},\widetilde{M})}+\Vert f \Vert_{L^p(\widetilde{M})})
\end{equation*}
where $C_{\chi}$ is a constant independent of $f$ and $g$.
\end{theorem}

To study the existence problem of subelliptic harmonic maps with potential, we also need some results about the heat kernel and subelliptic heat equation on compact sub-Riemannian manifolds. Let $K(x,y,t)$ be the heat kernel for $\Delta_H$ on a compact sub-Riemannian manifold $(M,H, g_H;g)$, that is
\begin{equation*}
\left\{\begin{gathered}
(\Delta_H-\frac{\partial}{\partial t})K(x,y,t)=0 \\
\underset{t\to 0}{\text{lim}}K(x,y,t)=\delta_x(y)
\end{gathered}\right.
\end{equation*}

According to \cite{MR2154760},\cite{Baudoin2018GeometricIO},\cite{MR755001} and \cite{MR862049}, we know that $K(x,y,t)$ exists. Some basic properties of $K(x,y,t)$ are listed as follows

(1) $K(x,y,t)\in C^{\infty}(M\times M\times \mathbb{R}^+)$;

(2) $K(x,y,t)=K(y,x,t)$ for $x,y\in M$ and $t>0$;

(3) $K(x,y,t)>0$ for $x,y\in M$ and $t>0$;

(4) $\int_M K(x,y,t)\mathrm{d}v_g=1$ for any $x\in M$;

(5) $K(x,y,t+s)=\int_M K(x,z,t)K(y,z,s)\,\mathrm{d}v_g(z)$.

\begin{lemma}\label{thm:subelliptic heat kernel gradient}(\cite[Lemma 3.3]{MR4236537})
For any $\beta\in (0,\frac 1 2)$, there exists $C_{\beta}> 0$ such that
\begin{equation*}
\int^t_0\int_M\vert\nabla^H_x K(x,y,s)\vert \,\mathrm{d}v_g(y) \,\mathrm{d}s \le C_{\beta} t^{\beta}
\end{equation*}
for $0<t<R_0$ for some positive constant $R_0$.
\end{lemma}

The following lemma gives both a maximum principle,
and a mean value type inequality for subsolutions of the subelliptic heat equation.

\begin{lemma}\label{thm:subelliptic maximum principle}(\cite[Lemma 3.4]{MR4236537})
Let $M$ be a compact sub-Riemannian manifold. Suppose $\phi$ is a subsolution of the subelliptic heat equation satisfying
\begin{equation*}
(\Delta_H-\frac{\partial}{\partial t})\phi\ge 0
\end{equation*}
on $M\times[0,T)$ with initial condition $\phi(x,0)=\phi_0(x)$ for any $x\in M$. Then
\begin{equation*}
\underset M{\text{sup}}\,\phi(x,t)\le \underset M{\text{sup}}\,\phi_0(x).
\end{equation*}
Furthermore, if $\phi(x,t)$ is nonnegative, then there exist a constant $B$ and an integer $Q$ such that
\begin{equation*}
\underset {x\in M}{\text{sup}}\,\phi(x,t)\le Bt^{-\frac Q 2}\int_M\phi_0(y) \,\mathrm{d}v_g(y)
\end{equation*}
for $0<t<min(R^2_0,T)$, where $R_0$ is as in Lemma \ref{thm:subelliptic heat kernel gradient}.
\end{lemma}

Choose an adapted frame field $\lbrace e_A\rbrace_{A=1}^{m+d}$ in $(M,H,g_H;g)$, and denote its dual frame field by ${\lbrace\omega^A\rbrace}_{A=1}^{m+d}$. Henceforth, we will make use of the following convention on the ranges of induces in $M$:
\begin{equation*}
\begin{split}
1\le A,B,C,\ldots,&\le m+d;\quad 1\le i,j,k,\ldots,\le m;\\
m+1&\le\alpha,\beta,\gamma,\ldots,\le m+d,
\end{split}
\end{equation*}
and the Einstein summation convention.

Let $(N,h)$ be a Riemannian manifold with Riemannian connection $\widetilde\nabla$. We choose an orthonormal frame field $\lbrace\widetilde e_I\rbrace_{I=1,\ldots,n}$ in $(N,h)$ and let $\lbrace\widetilde\omega^I\rbrace$ be its dual frame field. We will make use of the following convention on the ranges of indices in $N$:
\begin{equation*}
I,J,K=1,\ldots,n.
\end{equation*}

For a smooth map $f:M\to N$, we have a connection $\nabla^{\mathfrak B}\otimes \widetilde\nabla^f$ in $T^*M\otimes f^{-1}TN$, where $\widetilde\nabla^f$ is the pull-back connection of $\widetilde\nabla$. Then we can define the second fundamental form with respect to the data $(\nabla^{\mathfrak{B}}, \widetilde\nabla^f)$ as follows
\begin{equation*}
\beta(f;\nabla^{\mathfrak{B}},\widetilde\nabla)(X,Y)=\widetilde\nabla^f_Y \mathrm{d}f(X)-\mathrm{d}f(\nabla^{\mathfrak{B}}_Y X).
\end{equation*}
Using the frame fields in $M$ and $N$, the differential $\mathrm{d}f$ and the second fundamental form $\beta$ can be written as
\begin{equation*}
\mathrm{d}f=f_A^I\omega^A\otimes\widetilde e_I,
\end{equation*}
and
\begin{equation*}
\beta=f^I_{AB}\omega^A\otimes\omega^B\otimes\widetilde e_I
\end{equation*}
respectively.
Apart from the differential $\mathrm{d}f$, we also consider two partial differentials $ \mathrm{d}f_H= \mathrm{d}f\vert_H \in\Gamma(H^*\otimes f^{-1}TN)$ and $\mathrm{d}f_V= \mathrm{d}f\vert_V \in\Gamma(V^*\otimes f^{-1}TN)$. By definition, we get
\begin{equation*}
\vert\mathrm{d}f_H\vert^2=(f^I_i)^2,\quad \vert \mathrm{d}f_V\vert^2=(f^I_{\alpha})^2,\quad \vert\mathrm{d}f\vert^2=(f^I_A)^2.
\end{equation*}
Set
\begin{equation*}
e_H(f)=\frac 1 2\vert\mathrm{d}f_H\vert^2,\quad e_V(f)=\frac 1 2\vert  \mathrm{d}f_V\vert^2,\quad e(f)=\frac 1 2\vert\mathrm{d}f\vert^2.
\end{equation*}
Then we can define the following two partial energies:
\begin{equation*}
\begin{split}
E_H(f)&=\int_M e_H(f)\,\mathrm{d}v_g=\frac 1 2\int_M\langle \,\mathrm{d}f(e_i),df(e_i)\rangle \,\mathrm{d}v_g, \\
E_V(f)&=\int_M e_V(f)\,\mathrm{d}v_g=\frac 1 2\int_M\langle \,\mathrm{d}f(e_{\alpha}),\,\mathrm{d}f(e_{\alpha})\rangle dv_g.
\end{split}
\end{equation*}                                                                                                                                                                       The partial energies $E_H(f)$ and $E_V(f)$ are referred to as horizontal and vertical energies respectively. Obviously the usual Dirichlet energy $E(f)$ satisfies
\begin{equation*}
E(f)=E_H(f)+E_V(f).
\end{equation*}
For any potential function $G\in C^{\infty}(N)$, we introduce the following energies:
\begin{equation*}
\begin{split}
E_P(f)&=\int_M -G(f) \,\mathrm{d}v_g, \\                                                                                                                                                           E_G(f)&=E_H(f)+E_P(f)=\int_M [e_H(f)-G(f)]\,\mathrm{d}v_g.
\end{split}
\end{equation*}
We call energies $E_P(f)$ and $E_G(f)$ potential energy and horizontal energy with potential $G$ respectively.
\begin{definition}
A map $f:(M,H,g_H;g)\to (N,h)$ is called a subelliptic harmonic map with potential $G$ if it is a critical point of the energy $E_G(f)$.
\end{definition}
For the purpose of deriving the Euler-Lagrange equation for $E_G$, we study a variation of $\lbrace f_t\rbrace_{\vert t\vert <\epsilon}$, which is a family of maps from $(M,H,g_H;g)$ to $(N,h)$ with $f_0=f$ and $\left.\frac{\partial f_t}{\partial t}\right|_{t=0}=\nu\in\Gamma (f^{-1}TN)$. Since $\widetilde\nabla$ is torsion-free, using divergence theorem on the compact manifold $M$, a computation gives
\begin{equation}\label{eq:first derivative f}
\begin{split}
\frac d {dt}E_G(f_t)\vert_{t=0}
&=\int_M\langle\widetilde\nabla^{f_t}_{\frac{\partial}{\partial t}}\mathrm{d}f_t(e_i),\mathrm{d}f_t(e_i)\rangle \,\mathrm{d}v_g\vert_{t=0}-\int_M\langle\nu,\widetilde\nabla G(f)\rangle\,\mathrm{d}v_g \\
&=\int_M\langle\widetilde\nabla^f_{e_i}\nu,\,\mathrm{d}f(e_i)\rangle \,\mathrm{d}v_g-\int_M\langle\nu,\widetilde\nabla G(f)\rangle \,\mathrm{d}v_g  \\
&=\int_M\langle\nu,\,\mathrm{d}f(\zeta)\rangle \,\mathrm{d}v_g-\int_M\langle\nu,\beta(e_i,e_i)\rangle \,\mathrm{d}v_g-\int_M\langle\nu,\widetilde\nabla G(f)\rangle \,\mathrm{d}v_g \\
&=-\int_M\langle\nu,\tau(f)\rangle \,\mathrm{d}v_g
\end{split}
\end{equation}
where
\begin{equation*}
\tau(f)=\tau_H(f)+\widetilde\nabla G=\beta(e_i,e_i)-\mathrm{d}f(\zeta)+\widetilde\nabla G
\end{equation*}
is called the subelliptic tension field of $f$ with potential $G$. Consequently, we have the following equivalent characterization of subelliptic harmonic maps with potential $G$.
\begin{prop}
A map $f:(M,H,g_H;G)\to (N,h)$ is a subelliptic harmonic map with potential $G$ if and only if it satisfies the Euler-Lagrange equation
\begin{equation}\label{eq:euler-lagrange}
\tau(f)=0.
\end{equation}
\end{prop}

In order to solve (\ref{eq:euler-lagrange}), we follow Eells-Sampson's idea to deform a given smooth map $\bar u :M\to N$ along the gradient flow of the energy $E_G$. This is equivalent to solving the following subelliptic heat flow with potential $G$
\begin{equation}\label{eq:heat flow f}
\left\{\begin{gathered}
\frac{\partial f}{\partial t}=\tau(f)\\
\left. f\right|_{t=0}=\bar u
\end{gathered}\right.
\end{equation}
where $\tau(f(\cdot,t))$ is the subelliptic tension field with potential $G$ of $f(\cdot,t):(M,H,g_H;g)\to (N,h)$.

Before proving the existence theory, we intend to give the explicit formulations for both (\ref{eq:euler-lagrange}) and (\ref{eq:heat flow f}). According to the Nash embedding theorem, the isometric embedding $\mathfrak J:(N,h,\widetilde{\nabla})\to(\mathbb R^K,g_E,D)$ in some Euclidean space always exists, where $g_E$ denotes the standard Euclidean metric, $\widetilde{\nabla}$ and $D$ are the Riemannian connections of $(N,h)$ and $(\mathbb R^K,g_E)$ respectively. We also suppose that the potential $G$ is the restriction of some smooth function $\bar G:\mathbb R^K\to \mathbb R$ to $N$. Applying the composition formula for second fundamental forms \cite[page 14]{MR703510} to the maps $f:(M,\nabla^{\mathfrak B})\to(N,\widetilde{\nabla})$ and $\mathfrak J:(N,\widetilde{\nabla})\to(\mathbb R^K,D)$, we have
\begin{equation*}
\beta(\mathfrak J\circ f;\nabla^{\mathfrak B},D)(\cdot,\cdot)=\mathrm{d}\mathfrak J(\beta(f;\nabla^{\mathfrak B},\widetilde{\nabla})(\cdot,\cdot))+\beta(\mathfrak J;\widetilde{\nabla},D)(\mathrm{d}f(\cdot),\mathrm{d}f(\cdot)).
\end{equation*}
To simplify the notation, we identify $N$ with $\mathfrak J(N)$, and denote $\mathfrak J\circ f$  by $u$. Note that $u$ is a map from $M$ to $\mathbb R^K$.
Then, we get
\begin{equation}\label{eq:embeding flow formula 1}
\tau_H(u;\nabla^\mathfrak B,D)-tr_g\beta(\mathfrak J;\widetilde \nabla,D)(\mathrm{d}f_H,\mathrm{d}f_H)=\mathrm{d}\mathfrak J(\tau_H(f)).
\end{equation}

For the submanifold $N$, there exists a tubular neighborhood $B(N)$ of $N$ in $\mathbb R^K$ and a natural projection map
$\Pi:B(N)\to N$ which is a submersion over $N$. In fact, the map $\Pi$ maps any point in $B(N)$ to its closest point in $N$. Obviously its differential $\mathrm{d}\Pi:T_y\mathbb R^K\to T_y\mathbb R^K$ at a point $y\in N$ is given by the identity map when restricted to the tangent space $TN$ of $N$ and maps all the normal vectors to $N$ to the zero vector. Since $\Pi\circ\mathfrak J=\mathfrak J:N\hookrightarrow \mathbb R^K$ and $\beta(\mathfrak J;\widetilde\nabla,D)$ is normal to $N$, we have
\begin{equation*}
\beta(\mathfrak J;\widetilde\nabla,D)=\beta(\Pi;D,D)(\mathrm{d}\mathfrak J,\mathrm{d}\mathfrak J).
\end{equation*}
Choose the natural Euclidean coordinate system $\lbrace y^a\rbrace_{1\le a\le K}$ and set $u^a=y^a\circ u, \Pi^a=y^a\circ\Pi$. According to the above argument, we obtain
\begin{equation}\label{eq:euclidean coordinate 1}
\tau_H(u;\nabla^{\mathfrak B},D)=\Delta_H u^a\frac\partial{\partial y^a}
\end{equation}
and
\begin{equation}\label{eq:euclidean coordinate 2}
\begin{split}
tr_g\beta(\mathfrak J;\widetilde\nabla,D)(\mathrm{d}f_H,\mathrm{d}f_H)&=tr_g\beta(\Pi;D,D)(\mathrm{d}u_H,\mathrm{d}u_H) \\
&=\Pi^a_{bc}\langle\nabla^H u^b,\nabla^H u^c\rangle\frac\partial{\partial y^a}
\end{split}
\end{equation}
where $\Pi^a_{bc}=\frac{\partial^2 \Pi^a}{\partial y^b\partial y^c}$.
Note that
\begin{equation}\label{eq:embeding flow formula 2}
\mathrm{d}\mathfrak J(\tau(f))=\mathrm{d}\mathfrak J(\tau_H(f))+\mathrm{d}\Pi D (\bar G).
\end{equation}
Consequently, (\ref{eq:embeding flow formula 1}), (\ref{eq:euclidean coordinate 1}), (\ref{eq:euclidean coordinate 2}) (\ref{eq:embeding flow formula 2}) show that
\begin{equation}\label{eq:euclidean embeding flow}
\mathrm{d}\mathfrak J(\tau(f))=\left(\Delta_H u^a-\Pi^a_{bc}\langle\nabla^H u^b,\nabla^H u^c\rangle+\Pi^a_b(D\,\bar G^b)\right)\frac\partial{\partial y^a}.
\end{equation}
Therefore $f$ is a subelliptic harmonic map with potential $G$ if and only if $u=(u^a):M\to \mathbb R^K$ solves
\begin{equation*}
\Delta_H u^a-\Pi^a_{bc}\langle\nabla^H u^b,\nabla^H u^c\rangle+\Pi^a_b(D\,\bar G^b)=0, 1\le a,b,c\le K.
\end{equation*}
Based on the above explicit formulation for $\tau(u)$, we can establish the fact that the following system is equivalent to (\ref{eq:heat flow f}):
\begin{equation}\label{eq:tubular heat flow}
\left\{\begin{gathered}
\Delta_H u^a-\frac{\partial u^a}{\partial t}=\Pi^a_{bc}\langle\nabla^H u^b,\nabla^H u^c\rangle-\Pi^a_b(D\,\bar G^b(u)) \\
u^a(x,0)=\bar u^a(x), \qquad 1\le a,b,c\le K
\end{gathered}\right.
\end{equation}
whenever $\text{Im} u\subset B(N)$ for $t\in [0,T)$, where $\bar u^a=y^a\circ \bar u$.

To this end, we need to show $u(x,t)\in N$ for all $(x,t)\in M\times[0,T)$, where $[0,T)$ is the maximal existence domain of $u$.
Define a map $\rho:B(N)\to \mathbb{R}^K$ by
\begin{equation}\label{eq:definition of normal rho}
\rho(y)=y-\Pi(y),\quad y\in B(N).
\end{equation}
It is easy to see $\rho(y)$ is normal to $N$ and $\rho(y)=0$ is equivalent to $y\in N$.
Differentiating both sides of (\ref{eq:definition of normal rho}) twice, we have
\begin{equation*}
\rho^a_b=\delta^a_b-\Pi^a_b
\end{equation*}
and
\begin{equation*}
\rho^a_{bc}=-\Pi^a_{bc}
\end{equation*}
where $\rho^a_b=\frac{\partial\rho^a}{\partial y^b}$, $\Pi^a_b =\frac{\partial\Pi^a}{\partial y^b}$ and $\rho^a_{bc}=\frac{\partial \rho^a}{\partial y^b\partial y^c}$. Using the composition formula for second fundamental forms \cite{MR703510} and equation (\ref{eq:tubular heat flow}), we have
\begin{equation}\label{eq:normal part}
\begin{split}
(\Delta_H\rho(u))^a&=\rho^a_b\Delta_H u^b+\rho^a_{bc}\langle\nabla^H u^b,\nabla^H u^c\rangle\\
&=\Delta_H u^a-\Pi^a_b\Delta_H u^b-\Pi^a_{bc}\langle\nabla^H u^b,\nabla^H u^c\rangle\\
&=\rho^a_b\frac{\partial u^b}{\partial t}+\Pi^a_b(\frac{\partial u^b}{\partial t}-\Delta_H u^b+D\,\bar G^b).
\end{split}
\end{equation}
Note that $d\Pi(\frac{\partial u}{\partial t}-\Delta_H u+D\,\bar G)$ is tangent to $N$ and $\rho(u)$ is normal to $N$, then from (\ref{eq:normal part}) we obtain
\begin{equation*}
\rho^a(u)(\Delta_H\rho(u))^a=\rho^a(u)\rho^a_b(u)\frac{\partial u^b}{\partial t}.
\end{equation*}
In terms of divergence theorem and (\ref{eq:normal part}), we deduce that
\begin{equation*}
\begin{split}
\frac{\partial}{\partial t}\int_M(\rho^a(u))^2\mathrm{d}v_g&=2\int_M\rho^a(u)\rho^a_b(u)\frac{\partial u^b}{\partial t}\mathrm{d}v_g \\
&=2\int_M\rho^a(u)(\Delta_H\rho(u))^a\mathrm{d}v_g \\
&=-2\int_M\vert\nabla_H\rho(u)\vert^2\mathrm{d}v_g \\
&\le 0.
\end{split}
\end{equation*}
Thus,
\begin{equation*}
\int_M\vert\rho(u(x,t))\vert^2\mathrm{d}v_g
\end{equation*}
is non-increasing in $t$. In particular, if $\rho(u(x,0))=0$, then $\rho(u(x,t))=0$ for all $t\in [0,T)$, which means $u(x,t)\in N$ for all $(x,t)\in M\times[0,T)$.

\section{short time existence}\label{sec:flows}
In this section, we will establish the short-time existence result of (\ref{eq:heat flow u}).
To achieve this, we need the following Bochner type inequality for $e(f)$.

\begin{lemma}\label{thm:subelliptic bochner formula}
Let $(M,H,g_H;g)$ be a compact sub-Riemannian manifold and let $(N,h)$ be a Riemannian manifold with non-positive sectional curvature. Suppose that $f:M\to N$ is a smooth map. Set $\tau^I_H=f^I_{kk}-\zeta^k f^I_k$ and $\tau^I=\tau^I_H+[(\widetilde\nabla G)(f)]^I$. Then one has
\begin{equation}
\begin{split}
\Delta_H e(f)-f^I_i\tau^I_{,i}-f^I_{\alpha}\tau^I_{,\alpha}&\ge -C_{\epsilon}e_H(f)-\epsilon e_V(f)+(f^I_{ik})^2 \\
&\quad+\frac{1}{2}(f^I_{\alpha k})^2-\text{Hess}\,G(f_i,f_i)-\text{Hess}\,G(f_{\alpha},f_{\alpha})
\end{split}
\end{equation}
for any given $\epsilon>0$, where $C_{\epsilon}$ is a positive number depending only on $\epsilon$ and
\begin{equation*}
\underset{M,i,j,k,\alpha}{\text{sup}}\lbrace\vert\zeta^k\vert,\vert\zeta^k_{,i}\vert,\vert\zeta^k_{,\alpha}\vert,\vert T^{\alpha}_{ij}\vert,\vert T^{\alpha}_{ij,k}\vert,\vert R^j_{kik}\vert,\vert R^j_{k\alpha k}\vert\rbrace.
\end{equation*}
In particular, we have
\begin{equation}\label{eq:bochner formula}
\Delta_H e(f)-f^I_i\tau^I_{\,i}-f^I_{\alpha}\tau^I_{\,\alpha}\ge -C_{\epsilon}e(f)-\text{Hess}\,G(f_i,f_i)-\text{Hess}\,G(f_{\alpha},f_{\alpha}).
\end{equation}
\end{lemma}

\begin{proof}[Proof of Lemma \ref{thm:subelliptic bochner formula}]
Let $T=\frac 1 2(T^{A}_{BC}\omega^B\wedge\omega^C)\otimes e_{A}$ and $\Omega^A_B=\frac 1 2 R^A_{BCD}\omega^C\wedge\omega^D$ be the torsion and curvature of $\nabla^{\mathfrak B}$ respectively and let $\widetilde\Omega^I_J=\frac 1 2 \widetilde R^I_{JKL}\widetilde\omega^K\wedge\widetilde\omega^L$ be the curvature of $\widetilde\nabla$.
Denote covariant derivative of $\zeta^k$ and $T^{A}_{BC}$ by $\zeta^k_{,A}$ and $T^{A}_{BC,D}$ respectively.
From \cite{MR4236537}, we know
\begin{equation*}
\begin{split}
\Delta_H e(f)&=\Delta_H e_H(f)+\Delta_H e_V(f)\\
&=(f^I_{ik})^2+f^I_if^I_{ikk}-\zeta^kf^I_if^I_{ik}+(f^I_{\alpha k})^2+f^I_{\alpha}f^I_{\alpha kk}\\
&=(f^I_{ik})^2+(f^I_{\alpha k})^2+f^I_i\tau^I_{H,i}+f^I_{\alpha}\tau^I_{H,\alpha}+f^I_i\zeta^k_{,k}f^I_k \\
&\quad+\zeta^k f^I_i f^I_{\alpha}T^{\alpha}_{ki}+f^I_i f^I_j R^j_{kik}+2f^I_{i}f^I_{\alpha k}T^{\alpha}_{ik}-f^I_i f^K_k\widetilde R^I_{KJL}f^J_i f^L_k \\
&\quad+f^I_i f^I_{\alpha}T^{\alpha}_{ik,k}+f^I_{\alpha}\zeta^k_{,\alpha}f^I_k+f^I_{\alpha}f^I_j R^j_{k\alpha k}-f^I_{\alpha}f^K_k\widetilde R^I_{KJL}f^J_{\alpha}f^L_k.
\end{split}
\end{equation*}
Since $\tau^I=\tau^I_H+[(\widetilde\nabla G)(f)]^I$, then
\begin{equation*}
\begin{split}
\Delta_H e(f)&=(f^I_{ik})^2+(f^I_{\alpha k})^2+f^I_i\tau^I_{,i}+f^I_{\alpha}\tau^I_{,\alpha}-f^I_i G_{IJ}f^J_i-f^I_{\alpha}G_{IJ}f^J_{\alpha}+f^I_i\zeta^k_{,k}f^I_k \\
&\quad+\zeta^k f^I_i f^I_{\alpha}T^{\alpha}_{ki}+f^I_i f^I_j R^j_{kik}+2f^I_{i}f^I_{\alpha k}T^{\alpha}_{ik}-f^I_i f^K_k\widetilde R^I_{KJL}f^J_i f^L_k \\
&\quad+f^I_i f^I_{\alpha}T^{\alpha}_{ik,k}+f^I_{\alpha}\zeta^k_{,\alpha}f^I_k+f^I_{\alpha}f^I_j R^j_{k\alpha k}-f^I_{\alpha}f^K_k\widetilde R^I_{KJL}f^J_{\alpha}f^L_k.
\end{split}
\end{equation*}
where $\text{Hess}\,G=(G_{IJ})$.
Using Schwarz inequality and curvature assumption of $N$, we have that
\begin{equation*}
\begin{split}
&f_i^I \zeta^k_{,i}f^I_k+f_i^If_j^IR_{kik}^j\ge-C_1e_H(f),\\
&\zeta^kf_i^If^l_{\alpha}T^{\alpha}_{ki}+f_i^If^I_{\alpha}T^{\alpha}_{ik,k}+f^I_{\alpha}\zeta^k
_{,\alpha}f^I_k+f^I_{\alpha}f_j^IR^j_{k\alpha k}\ge-C_2(\epsilon)e_H(f)-\epsilon e_V (f),\\
&2f_i^If^I_{\alpha k}T^{\alpha}_{ik}\ge-C_3e_H(f)-\frac{1}{2}(f^I_{\alpha k})^2,\\
&f_i^If_k^K \widetilde R^I_{KJL}f_i^Jf_k^L+f^I_{\alpha}f_k^K \widetilde R^I_{KJL}f_{\alpha}^Jf_k^L \le 0.
\end{split}
\end{equation*}
These estimates give (\ref{eq:bochner formula}).
\end{proof}

\begin{prop}\label{thm:short time existence}
Let $(M^{m+d},H,g_H;g)$ be a compact sub-Riemannian manifold, and $(N,h)$ be a compact Riemannian manifold. The heat flow (\ref{eq:heat flow u}) admits a unique smooth solution defined
on a maximal existence domain $M\times[0, T)$.
\end{prop}

\begin{proof}[Proof of Proposition \ref{thm:short time existence}]
Writing $u=(u^a(x,t))_{1\le a\le K}$, the subelliptic harmonic map heat flow with potential $G$ becomes
\begin{equation*}
\left\{\begin{gathered}
(\Delta_H -\frac{\partial}{\partial t}) u=F(x,t) \\
u(x,0)=\bar u(x)
\end{gathered}\right.
\end{equation*}
where $F(x,t)=\left(\Pi^a_{bc}\langle\nabla^H u^b,\nabla^H u^c\rangle-\Pi^a_b(D\,\bar G^b(u))\right)$.

By Duhamel's principle, a sequence of approximate solutions can be defined inductively as follows:
\begin{equation}\label{eq:inductive solution}
\begin{split}
u_0(x,t)&=\int_M K(x,y,t)\bar u(y)\,\mathrm{d}v_g(y)\\
u_k(x,t)&=u_0(x,t)-\int^t_0\int_M K(x,y,t-s)F_{k-1}(y,s)\,\mathrm{d}v_g(y)\,\mathrm{d}s
\end{split}
\end{equation}
where
\begin{equation}\label{eq:inductive function}
F_{k-1}(y,s)=\left((\Pi^a_{bc}\langle\nabla^H u^b_{k-1},\nabla^H u^c_{k-1}\rangle)-\Pi^a_b(D\,\bar G^b(u_{k-1}))\right),\quad k\ge 1.
\end{equation}
It is clear that $u_0$ and $u_k:M\to \mathbb{R}^{K}$ solve respectively
\begin{equation*}
\left\{\begin{gathered}
(\Delta_H -\frac{\partial}{\partial t}) u_0=0 \\
u_0(x,0)=\bar u(x)
\end{gathered}\right.
\end{equation*}
and
\begin{equation*}
\left\{\begin{gathered}
(\Delta_H -\frac{\partial}{\partial t}) u_k=F_{k-1}(x,t) \\
u_k(x,0)=\bar u(x),\qquad k\ge 1.
\end{gathered}\right.
\end{equation*}
Set
\begin{equation*}
\begin{split}
\Lambda &=\underset{B(N),a,b,c,d}{\text{sup}}\lbrace{\vert\Pi^a_{b}\vert,\vert\Pi^a_{bc}\vert,\vert\frac{\partial\Pi^a_{bc}}{\partial y^d}\vert}\rbrace \\
P&=\underset{B(N),I,J}{\text{sup}}\lbrace\vert D\,\bar G\vert, \vert \bar G_{IJ}\vert\rbrace,
\end{split}
\end{equation*}
where $(y^1,\dotsc,y^K)$ are coordinates of $\mathbb{R}^K$, $B(N)$ is the tubular neighborhood of $N$ and $\Pi$ is the closest point map over $B(N)$. We denote
\begin{equation}\label{eq:inductive energy}
p_{k-1}(t)=\underset{M\times[0,t)}{\text{sup}}\sqrt{e_H(u_{k-1})},\qquad k\ge 1,
\end{equation}
which is a non-decreasing function of $t$. From (\ref{eq:inductive function}) and (\ref{eq:inductive energy}), we obtain
\begin{equation}\label{eq:inductive function estimate}
\underset{M\times[0,t)}{\text{sup}}\lvert F_{k-1}(x,s)\rvert \le \Lambda(p^2_{k-1}(t)+P).
\end{equation}
Since $\int_M K(x,y,t)dy=1$, we find that
\begin{equation}\label{eq:initial estimate}
\vert u_0\vert\le\Vert\bar u\Vert_{C^0}=\underset{x\in M}{\text{sup}}\sqrt{\sum\limits_{a=1}^K(\bar u^a(x))^2},
\end{equation}
where $\Vert\cdot\Vert_{C^0}$ denotes the $C^0$-norm of functions or tensor fields on $M$. Using (\ref{eq:inductive solution}) (\ref{eq:inductive function}) (\ref{eq:inductive function estimate}) and (\ref{eq:initial estimate}), we deduce that
\begin{equation}\label{eq:distance estimate}
\begin{split}
\vert u_k-u_0\vert&\le\Lambda(p^2_{k-1}+P)t \\
\vert u_k\vert&\le\Lambda t(p^2_{k-1}+P)+\Vert\bar u\Vert_{C^0}.
\end{split}
\end{equation}
Note that $\tau_H(u_0)=\Delta_H u_0$ for the map $u_0:M\to \mathbb{R}^K$. By Lemma \ref{thm:subelliptic bochner formula}, we derive that
\begin{equation*}
(\Delta_H-\frac{\partial}{\partial t})(e^{-Ct}e(u_0))\ge 0.
\end{equation*}
As a result, Lemma \ref{thm:subelliptic maximum principle} gives that
\begin{equation*}
e^{-Ct}e(u_0)\le e(\bar u)
\end{equation*}
and thus
\begin{equation}\label{eq:initial energy estimate}
p_0(t)\le\sqrt{e^{Ct}e(\bar u)}.
\end{equation}
In view of Lemma \ref{thm:subelliptic heat kernel gradient}, (\ref{eq:inductive solution}) and (\ref{eq:inductive function estimate}), we have
\begin{equation*}
\begin{split}
\vert\nabla^H_x u_k(x,t)\vert&\le\vert\nabla^H_x u_0\vert+\int^t_0\int_M\vert\nabla^H_x K(x,y,t-s)\vert\cdot\vert F_{k-1}(y,s)\vert \,\mathrm{d}v_g(y) \\
&\le\vert\nabla^H_x u_0\vert+C_1\Lambda(p^2_{k-1}(t)+P)t^{\beta},
\end{split}
\end{equation*}
which yields
\begin{equation}\label{eq:inductive inequality}
\begin{split}
p_k(t)&\le C_1\Lambda (p^2_{k-1}(t)+P)t^{\beta}+p_0(t) \\
& \le C_1\Lambda (p_{k-1}(t)+\sqrt P)^2t^{\beta}+p_0(t).
\end{split}
\end{equation}
Choosing $\delta$ sufficiently small, it follows from (\ref{eq:initial energy estimate}) that
\begin{equation*}
C_1\Lambda\delta^{\beta}(p_0(\delta)+\sqrt P)\le C_1\Lambda\delta^{\beta}(\sqrt{e^{C\delta}e(\bar u)}+\sqrt{P})\le \frac{\epsilon}4,
\end{equation*}
for any $0<\epsilon<1$.
In terms of (\ref{eq:inductive inequality}), we get inductively
\begin{equation*}
\begin{split}
C_1\Lambda\delta^{\beta}(p_k(\delta)+\sqrt P)&\le\Big(C_1\Lambda\delta^{\beta}(p_{k-1}(\delta)+\sqrt P)\Big)^2+C_1\Lambda\delta^{\beta}(p_0(\delta)+\sqrt P) \\
&\le\frac{\epsilon}4+\frac{\epsilon}4 =\frac{\epsilon}2,
\end{split}
\end{equation*}
and thus
\begin{equation}\label{eq:energy estimate variant}
C_1\Lambda\delta^{\beta}p_k(\delta)\le C_1\Lambda\delta^{\beta}(p_k(\delta)+\sqrt P)\le\frac{\epsilon}2.
\end{equation}
Consequently
\begin{equation}\label{eq:inductive energy estimate}
p_k(\delta)\le C_2\epsilon\delta^{-\beta}.
\end{equation}
Let us introduce the following space of functions
\begin{equation*}
C^1_H(M,\mathbb{R}^K)=\lbrace f:M\to \mathbb{R}^K\vert f\in C^0,\,\nabla^H f\in C^0\rbrace,
\end{equation*}
which is equipped with the norm
\begin{equation*}
\Vert f\Vert_{C^1_H}=\Vert f\Vert_{C^0}+\Vert\nabla^H f\Vert_{C^0}.
\end{equation*}
It is a fact that $(C^1_H(M,\mathbb{R}^K),\Vert\cdot\Vert_{C^1_H})$ is a Banach space. Combining (\ref{eq:distance estimate}) with (\ref{eq:inductive energy estimate}), one gets
\begin{equation*}
\Vert u_k\Vert_{C^1_H(M,R^K)}\le C_3(C_2,\epsilon,\delta).
\end{equation*}
In terms of (\ref{eq:distance estimate}) on $M\times[0,\delta)$ and using (\ref{eq:energy estimate variant}), one has
\begin{equation}\label{eq:sufficiently small}
\vert u_k(x,t)-u_0(x,t)\vert\le\Lambda\delta (p^2_{k-1}+P)(\delta)\le\Lambda\delta (p_{k-1}+\sqrt P)^2(\delta)\le\frac{\epsilon^2\delta^{1-2\beta}}{4C_1\Lambda}.
\end{equation}
If we choose a sufficiently small $\delta$, then the inequality (\ref{eq:energy estimate variant}) is valid. We notice that $1-2\beta> 0$ (see Lemma \ref{thm:subelliptic heat kernel gradient}). Hence (\ref{eq:sufficiently small}) suggests that all maps $u_k\,(k\ge 1)$ will map $M$ into $B(N)$ by choosing $\delta$ sufficiently small since $\vert u_0(x,t)-\bar u(x)\vert$ can be chosen to be sufficiently small for small $t$, owing to continuity of $u_0$.

The proof of the theorem will be complete, if we show that $\lbrace u_k(x,t)\rbrace$ is a Cauchy sequence in $C^1_H(M,\mathbb{R}^K)$, in the case that $t$ is sufficiently small. To this end, we investigate the following sequence
\begin{equation*}
X_k(t)=\underset{M\times[0,t)}{\text{sup}}\lbrace\vert u_k(x,s)-u_{k-1}(x,s)\vert+\vert\nabla^H_x u_k(x,s)-\nabla^H_x u_{k-1}(x,s)\vert\rbrace
\end{equation*}
which is non-decreasing in $t$. Note that
\begin{equation*}
\begin{split}
F_k(x,t)-F_{k-1}(x,t)
&=\Big(\Pi^a_{bc}(u_k)-\Pi^a_{bc}(u_{k-1})\Big)\langle\nabla^H u^b_k,\nabla^H u^c_k\rangle \\
&\quad+\Pi^a_{bc}(u_{k-1})\langle\nabla^H u^b_k-\nabla^H u^b_{k-1},\nabla^H u^c_k\rangle \\
&\quad+\Pi^a_{bc}(u_{k-1})\langle\nabla^H u^b_{k-1},\nabla^H u^c_k-\nabla^H u^c_{k-1}\rangle \\
&\quad+\Pi^a_b\Big(D\,\bar G^b(u_k)-D\,\bar G^b(u_{k-1})\Big).
\end{split}
\end{equation*}
Using (\ref{eq:inductive energy estimate}) and the following estimates
\begin{equation*}
\begin{split}
&\vert\Pi^a_{bc}(u_k)-\Pi^a_{bc}(u_{k-1})\vert\le\Lambda\vert u_k-u_{k-1}\vert, \\
&\vert D\,\bar G(u_k)-D\,\bar G(u_{k-1})\vert\le P\vert u_k-u_{k-1}\vert,
\end{split}
\end{equation*}
we find that
\begin{equation*}
\begin{split}
\underset{M\times[0,t]}{\text{sup}}\vert F_k(x,t)-F_{k-1}(x,t)\vert&\le C_4 X_k(t)(p^2_k(t)+p_k(t)+p_{k-1}(t)) \\
&\le C_5 X_k(t)
\end{split}
\end{equation*}
for any $t\le\delta$. As a result, we get
\begin{equation*}
\begin{split}
\vert u_k-u_{k-1}\vert&\le\int_0^t\int_M K(x,y,t-s)\vert F_{k-1}(y,s)-F_{k-2}(y,s)\vert \,\mathrm{d}v_g(s)\,\mathrm{d}s \\
&\le C_5 tX_{k-1}(t)
\end{split}
\end{equation*}
and
\begin{equation*}
\begin{split}
&\vert \nabla^H_x u_k-\nabla ^H_x u_{k-1}\vert \\
&\le\int_0^t\int_M \vert\nabla^H_x K(x,y,t-s)\vert\cdot\vert F_{k-1}(y,s)-F_{k-2}(y,s)\vert \,\mathrm{d}v_g(s)\,\mathrm{d}s \\
&\le C_6 t^{\beta}X_{k-1}(t),
\end{split}
\end{equation*}
which yield
\begin{equation}\label{eq:inductive norm inequality}
X_k(t)\le C_7t^{\beta}X_{k-1}(t)
\end{equation}
for $k\ge 2$. For $k=1$, using $t<1$, we obtain from (\ref{eq:inductive solution}) and (\ref{eq:initial energy estimate}) that
\begin{equation*}
\begin{split}
\vert u_1(x,t)-u_0(x,t)\vert&\le\int_0^t\int_M K(x,y,t-s)\vert F_0(y,s)\vert \,\mathrm{d}v_g(y)\,\mathrm{d}s \\
&\le t\Lambda p^2_0(t) \le t\Lambda (e^Ce(\bar u)+P)
\end{split}
\end{equation*}
and
\begin{equation*}
\begin{split}
\vert\nabla^H_x u_1(x,t)-\nabla^H_x u_0(x,t)\vert&\le\int_0^t\int_M \vert\nabla^H_x  K(x,y,t-s)\vert\cdot\vert F_0(y,s)\vert \,\mathrm{d}v_g(y)\,\mathrm{d}s \\
&\le C_1t^{\beta}\Lambda p^2_0(t) \le C_1t^{\beta}\Lambda (e^Ce(\bar u)+P).
\end{split}
\end{equation*}
It follows that
\begin{equation}\label{eq:initial norm estimate}
X_1(t)\le C_8(C_7t^{\beta})(e(\bar u)+P).
\end{equation}
By iterating (\ref{eq:inductive norm inequality}) and using (\ref{eq:initial norm estimate}), we have
\begin{equation}\label{eq:norm estimate}
X_k(t)\le C_8(C_7t^{\beta})^k (e(\bar u)+P).
\end{equation}
Choosing $\delta_0 $ sufficiently small such that $0 <\delta_0 \le\delta$ and $C_7\delta^{\beta}<1$, then (\ref{eq:norm estimate}) yields that for any $i\le j$
\begin{equation*}
\underset{[0,\delta_0]}{\text{sup}}\Vert u_i(\cdot,t)-u_j(\cdot,t)\Vert_{C^1_H(M)}
\le\sum\limits_{k=i+1}^{j} X_k(\delta_0)\le C_9 \sum\limits_{k=i+1}^{j} (C_7\delta^{\beta})^k,
\end{equation*}
which vanishes as $i,j\to\infty$. Therefore there exists $u\in C^0(M\times[0,\delta_0],B(N))$ with $u(\cdot,t)\in C^1_H(M,B(N))$ for each $t\in [0,\delta_0]$, such that $u_k\to u$ and $\nabla^H u_k\to\nabla^H u$ uniformly on $M\times[0,\delta_0]$. Thus
\begin{equation*}
F_k(x,t)\to F(x,t)=\Pi^a_{bc}(u)\langle\nabla^H u^b,\nabla^H u^c\rangle-\Pi^a_b(D\,\bar G^b(u))
\end{equation*}
and hence (\ref{eq:inductive solution}) implies that $u$ is given by
\begin{equation*}
u(x,t)=\int_M K(x,y,t)\varphi(y)dv_g(y)-\int^t_0\int_M K(x,y,t-s)F(y,s)\,\mathrm{d}v_g(y)\,\mathrm{d}s.
\end{equation*}
Clearly $u$ solves the subelliptic harmonic map heat flow with potential weakly. In view of Theorem \ref{thm:subelliptic regularity theory} and by a bootstrapping argument, we find that $u\in C^{\infty}(M\times(0,\delta_0),N)$ solves (\ref{eq:heat flow u}).

Next we will prove the solution of (\ref{eq:heat flow u}) is unique. Let $u$ and $v$ be solutions on $M\times[0,\delta)$ to (\ref{eq:tubular heat flow}) with the same initial condition: $u(x,0)=v(x,0)=\bar u$. Set $\Psi=\sum\nolimits^K_{a=1}(u^a-v^a)^2$, a similar computation as in the proof of \cite[Theorem 6.2]{MR4236537} shows
\begin{equation*}
(\Delta_H-\frac{\partial}{\partial t})(e^{-\widetilde Ct}\Psi)\ge 0,
\end{equation*}
where $\widetilde C(\delta_0,\Lambda,P,e_H(u),e_H(v))$ is a positive constant.
Then the uniqueness follows immediately from Lemma \ref{thm:subelliptic maximum principle}.
\end{proof}

\begin{remark}\label{thm:short time for complete}
When N is complete but not necessarily compact, there exists an open neighborhood $N'$ of $\bar u(M)$ with compact closure so that $N'$ can be embedded into $\mathbb{R}^K$ isometrically, since $\bar u$ is smooth and $M$ is compact. If necessary, by choosing a smaller neighborhood, we may assume that there exists a bounded tubular neighborhood $\widetilde N$ of $N'$ in $\mathbb{R}^K$  and the nearest point projection $\Pi : \widetilde N \to N$ can be extended smoothly to the whole $\mathbb{R}^K$ so that each $\Pi^a$
is compactly supported. Also, $\bar G\vert_{\widetilde N}$, which is the restriction of
$\bar G$ to $\widetilde N$, can be extended smoothly to a smooth function with compact support on $\mathbb{R}^K$, which we still denote by $\bar G$ for simplicity. Set
\begin{equation*}
\begin{split}
\Lambda &=\underset{\mathbb{R}^K,a,b,c,d}{\text{sup}}\lbrace{\vert\Pi^a_{b}\vert,\vert\Pi^a_{bc}\vert,\vert\frac{\partial\Pi^a_{bc}}{\partial y^d}\vert}\rbrace \\
P&=\underset{\mathbb{R}^K,I,J}{\text{sup}}\lbrace\vert D\,\bar G\vert, \vert \bar G_{IJ}\vert\rbrace,
\end{split}
\end{equation*}
then $\Lambda$ and $P$ are bounded.
Constructing a sequence of approximate solutions $u_k$ as we do in the proof of Proposition \ref{thm:short time existence}, we can get
\begin{equation*}
\vert u_k(x,t)-u_0(x,t)\vert\le\frac{\epsilon^2\delta^{1-2\beta}}{4C_1\Lambda}
\end{equation*}
which implies that all maps $u_k(k\ge 1)$ map $M$ into $\widetilde N$ for $t\in[0,\delta)$ by choosing $\delta$ sufficiently small. By showing that $\lbrace u_k(x,t)\rbrace$ is a Cauchy sequence in $C^1_H(M,\mathbb{R}^K)$, we can establish the short-time existence result of (\ref{eq:heat flow u}) in the case that $N$ is complete. (cf. \cite{MR1109619} for a similar discussion for harmonic map heat flows.)
\end{remark}

\section{long time existence}\label{sec:long}
In this section, we will apply a standard method to prove the long time existence.

\begin{proof}[Proof of Theorem \ref{thm:long time existence for complete}]
First, we assume that $N$ is a complete manifold and $\text{Hess}\, G\le C$. The short time existence of solution of (\ref{eq:heat flow u}) is a direct consequence of Proposition \ref{thm:short time existence} and Remark \ref{thm:short time for complete}. Let $u(x,t)$ be the solution, and let $[0,T)$ be its maximal existence domain. Suppose $T<+\infty$, we want to show:
\begin{equation}\label{eq:short time estimates}
\begin{split}
&(i)\qquad \vert \frac{\partial u(t)}{\partial t}\vert \le C(T),\\
&(ii)\qquad \vert \mathrm{d}u(t)\vert \le C(T),\\
&(iii)\quad d_{N}(u(t),\bar u) \le C(T),
\end{split}
\end{equation}
for some finite number $C(T)$ on $[0,T)$, where $d_N$ is the Riemannian distance function on $(N,h)$.

From \cite[Section 4]{MR4236537}, we know
\begin{equation*}
u_{At}=u_{tA}
\end{equation*}
and
\begin{equation*}
u_{AtB}-u_{ABt}=-u^K_A\widetilde R^I_{KJL}f^J_tf^L_B.
\end{equation*}
To prove (\ref{eq:short time estimates}), a simple computation gives
\begin{equation}\label{eq:time derivative estimate}
\begin{split}
(\frac{\partial}{\partial t}-\Delta_H)\vert \frac{\partial u(t)}{\partial t}\vert^2&=-2\vert\nabla^{H} \frac{\partial u(t)}{\partial t}\vert^2-2(\frac{\partial u(t)}{\partial t})_{kk}\frac{\partial u(t)}{\partial t}-2\zeta^k(\frac{\partial u(t)}{\partial t})_{k}\frac{\partial u(t)}{\partial t}\\
& +2\frac{\partial u(t)}{\partial t}\frac{\partial^2 u(t)}{\partial t^2} \\
&=-2\vert\nabla^{H} \frac{\partial u(t)}{\partial t}\vert^2+2\frac{\partial u(t)}{\partial t}\frac{\partial \tau_H(u)}{\partial t}+2\frac{\partial u(t)}{\partial t}\frac{\partial^2 u(t)}{\partial t^2} \\
& +2\langle\text{Riem}_{N}(du(t)(e_i),\frac{\partial u(t)}{\partial t})du(t)(e_i),\frac{\partial u(t)}{\partial t}\rangle\\
&=-2\vert\nabla^{H} \frac{\partial u(t)}{\partial t}\vert^2+2\text{Hess}\,G(\frac{\partial u(t)}{\partial t},\frac{\partial u(t)}{\partial t})\\
&+2\langle\text{Riem}_{N}(du(t)(e_i),\frac{\partial u(t)}{\partial t})du(t)(e_i),\frac{\partial u(t)}{\partial t}\rangle.
\end{split}
\end{equation}
Since sectional curvature $\text{Riem}_{N}\le 0$ and $\text{Hess}\,G\le C\cdot h$, then for some constant $C_1> 0$, we have
\begin{equation}\label{eq:time derivative subsolution}
(\frac{\partial}{\partial t}-\Delta_H)\vert \frac{\partial u(t)}{\partial t}\vert^2\le C_1 \vert \frac{\partial u(t)}{\partial t}\vert^2.
\end{equation}
From Lemma \ref{thm:subelliptic maximum principle}, we derive that
\begin{equation}\label{eq:first short time estimate}
\underset {x\in M}{\text{sup}}\,\vert \frac{\partial u}{\partial t}\vert^2(x,t)\le \underset {x\in M}{\text{sup}}\,e^{C_1 t} \vert \frac{\partial u}{\partial t}\vert^2(x,0)\le \underset {x\in M}{\text{sup}}\,e^{C_1 T} \vert \frac{\partial u}{\partial t}\vert^2(x,0),
\end{equation}
which proves (\ref{eq:short time estimates})(i).

The estimate (\ref{eq:short time estimates})(iii) follows from (\ref{eq:first short time estimate}), since
\begin{equation*}
d_{N}(u(t),\bar u)\le \int_0^t\vert\frac{\partial u(s)}{\partial s}\vert \,\mathrm{d}s\le \underset {x\in M}{\text{sup}}\, \vert \frac{\partial u(t)}{\partial t}\vert \int_{0}^{t} e^{\frac{C_1}{2} t} \,\mathrm{d}s\le \frac{2}{C_1} e^{\frac{C_1}{2} T} \underset {x\in M}{\text{sup}}\, \vert \frac{\partial u(0)}{\partial t}\vert.
\end{equation*}
Finally, using Lemma \ref{thm:subelliptic bochner formula}, we may deduce
\begin{equation*}
(\frac{\partial}{\partial t}-\Delta_H)e(u(t))\le C_{\epsilon}e(u(t))+\text{Hess}\,G(u_i(t),u_i(t))+\text{Hess}\,G(u_{\alpha}(t),u_{\alpha}(t)).
\end{equation*}
Since $\text{Hess}\,G\le C\cdot h$, we see that
\begin{equation}\label{eq:second short time estimate}
(\frac{\partial}{\partial t}-\Delta_H)e(u(t))\le C_2e(u(t)),
\end{equation}
for some constant $C_2>0$.
By Lemma \ref{thm:subelliptic maximum principle}, we get
\begin{equation*}
\underset {x\in M}{\text{sup}}\,e(u)(x,t)\le \underset {x\in M}{\text{sup}}\,e^{C_1 t} e(u)(x,0)\le \underset {x\in M}{\text{sup}}\,e^{C_1 T} e(u)(x,0).
\end{equation*}
The proof of (\ref{eq:short time estimates}) is achieved.
Since $\bar u(M)\subset N$ is compact and $(\ref{eq:short time estimates})(iii)$, we find that, for any $t \in [0,T)$, $u(t)(M) \subset N'$, where $N' \subset N$ is compact. Consider a sequence $T_i \to T$. Taking $u(\cdot,T_i)$ as a initial map, we may find a solution $u$ of (\ref{eq:euclidean heat flow}) on $[0,T+\delta')$ for some positive number $0<\delta'<\delta$ by Proposition \ref{thm:short time existence} if $i$ is sufficiently large.

From above discussion, it is easy to see that if $N$ is compact, then we have the long time existence for any $G\in C^{\infty}(N)$.
\end{proof}

\begin{remark}
If we take the potential function $G \equiv 0$, we get the following corollary
immediately.
\end{remark}

\begin{corollary}\label{thm:without potential}
Let $(M,H,g_H;g)$ be a compact sub-Riemannian manifold and let $(N,h)$ be a complete Riemannian manifold with non-positive sectional curvature. Then for any smooth map $\bar u : M \to N$, the following heat flow
\begin{equation}\label{eq:heat flow without potential}
\left\{\begin{gathered}
\frac{\partial u}{\partial t}=\tau_H(u)\\
\left.u\right|_{t=0}=\bar u
\end{gathered}\right.
\end{equation}
admits a unique
smooth solution defined on $M \times [0,+\infty)$.
\end{corollary}

\section{eells-sampson type theorem}\label{sec:eells}

In this section, we will establish the Eells-Sampson type theorem for the
subelliptic harmonic map heat flow with potential $G$. The only obstacle is to prove the convergence of $u$ at infinity. In order to prove that, we first need to show that $e(u)$ is uniformly bounded with respect to $t\in[0,\infty)$. The next lemma tells us that it is sufficient to estimate the upper bound of $E(u)$.

\begin{lemma}\label{thm:energy density}
Let $f:M\to N$ be a solution of the subelliptic harmonic map heat flow with potential $G$ on $[0,\delta)$. Suppose $(N,h)$ has non-positive sectional curvature and $\text{Hess}\,G\le C\cdot h$. Set $\alpha=min\lbrace R_0,\sqrt\delta\rbrace$, where $R_0$ is given by Lemma \ref{thm:subelliptic maximum principle}. Then
\begin{equation*}
e(f(\cdot,t))\le C(\epsilon_0)E(f(\cdot,t-\epsilon_0))
\end{equation*}
for $t\in[\epsilon_0,\delta)$, where $\epsilon_0$ is a fixed number in $(0,\frac{\alpha^2}2)$.
\end{lemma}
Using (\ref{eq:second short time estimate}), the proof of Lemma \ref{thm:energy density} is similar to \cite[Lemma 6.4]{MR4236537}, so we omit it.

From (\ref{eq:first derivative f}), it follows that
\begin{equation}\label{eq:first derivative u}
\frac{\mathrm{d}}{\mathrm{d}t}E_G(u(\cdot,t))=-\int_M\vert \tau(u(\cdot,t)) \vert^2 \mathrm{d}v_g \le 0,
\end{equation}
which implies $E_G(u(\cdot,t))\le E_G(\bar u)$, where $u(\cdot,t)$ is a solution of (\ref{eq:heat flow u}). When $\text{Im} u \subset N'$, where $N'$ is a compact subset of $N$, we have $\vert E_P(u(\cdot,t))\vert<(\text{max}_{N'}\vert G \vert)\text{vol}(M)$. Therefore it is enough to estimate $E_V(u(\cdot,t))$ for our purpose in the case that $\text{Im} u$ has compact closure $N'$.

Let: $S(V)\to M$ be the unit sphere bundle of the vertical bundle $V$. For any $v\in S(V)$, the $v$-component of $T(\cdot,\cdot)$ is given by $T^v(\cdot,\cdot)=\langle T(\cdot,\cdot),v\rangle$. Then we have a smooth function $\eta(v)=\frac 1 2\Vert T^v\Vert^2_g:S(V)\to \mathbb{R}$, given by
\begin{equation}\label{eq:eta}
\begin{split}
\eta(v)&=\sum\limits_{1\le i\le j\le m}(T^{\alpha}_{ij})^2\langle e_{\alpha},v\rangle^2 \\
&=\sum\limits_{1\le i\le j\le m}\langle [e_i,e_j],v\rangle^2.
\end{split}
\end{equation}

\begin{lemma}(\cite[Lemma 6.6]{MR4236537})\label{thm:eta}
H is $2$-step bracket generating if and only if $\eta(v)>0$ for each $v\in S(V)$.
\end{lemma}

\begin{lemma}\label{thm:step-2}
Let $(M,H,g_H;g)$ be a compact step-$2$ sub-Riemannian manifold and set $\eta_{min}=min_{v\in S(V)}\eta(v)$. Let $(N,h)$ be a compact Riemannian manifold with non-positive sectional curvature. Suppose $u:M\times[0,\delta)\to N$ is a solution of the subelliptic harmonic map heat flow with potential $G$ and $\text{Hess}\,G<\frac{\eta_{min}}{2}\cdot h$. Let $\epsilon$ be a fixed number with $0<\epsilon<\frac{1}{2}(\eta_{min}-2\lambda_G)$ in Lemma \ref{thm:subelliptic bochner formula}, where $\lambda_G$ is the minimum number such that $\text{Hess}\,G\le\lambda_G\cdot h$. Then, for any given $t_0\in(0,\delta)$, we have
\begin{equation*}
\begin{split}
E_V(u(\cdot,t))&\le E_V(u(\cdot,t_0))+\frac 2 {\eta_{min}-2\lambda_G}\Big(2\int_M\vert\tau(u(\cdot,t_0))\vert^2 \mathrm{d}v_g \\
&\quad+\int_M\vert\nabla G\vert^2 \mathrm{d}v_g+C(E_G(u(\cdot,t_0))+(\text{max}_{N}\vert G \vert)\text{vol}(M)\Big).
\end{split}
\end{equation*}
for any $t\in(t_0,\delta)$, where $C(\eta_{min},G)$ is a positive constant.
\end{lemma}
\begin{proof}[Proof]
Since $M$ is compact, $S(V)$ is compact. Hence there exists a point $v\in S(V)$ such that $\eta_{min}=\eta(v)$. Note that $(M, H)$ is a step-$2$ sub-Riemannian manifold, thus from Lemma \ref{thm:eta}, we have $\eta_{min}>0$. Let $\lambda_G$ be the minimum number such that $\text{Hess}\,G\le\lambda_G\cdot h$, and let $\epsilon$ be a fixed number with $0<\epsilon<\frac{1}{2}(\eta_{min}-2\lambda_G)$. From (4.17) in \cite{MR4236537}, we have
\begin{equation}\label{eq:torsion formula}
u^I_{ij}-u^I_{ji}=u^I_{\alpha}T^{\alpha}_{ij}.
\end{equation}
From (\ref{eq:eta}), (\ref{eq:torsion formula}) and Lemma \ref{thm:subelliptic bochner formula}, one has
\begin{equation}\label{eq:step-2 bochner}
\begin{split}
(\Delta_H-\frac{\partial}{\partial t})e(u)
\ge&-C_{\epsilon}e_H(u)-\epsilon e_V(u)+(u^I_{ik})^2+\frac{1}{2}(u^I_{\alpha k})^2\\
&-\text{Hess}\,G(u_i,u_i)-\text{Hess}\,G(u_{\alpha},u_{\alpha})\\
\ge&-C_{\epsilon}e_H(u)-\epsilon e_V(u)+\frac 1  2\sum\limits_I\sum\limits_{i<j}((u^I_{ij}+u^I_{ji})^2+(u^I_{ij}-u^I_{ji})^2) \\
&-\text{Hess}\,G(u_i,u_i)-\text{Hess}\,G(u_{\alpha},u_{\alpha})\\
\ge&-C_{\epsilon}e_H(u)-\epsilon e_V(u)+\frac 1  2\sum\limits_I\sum\limits_{\alpha}\sum\limits_{i<j}(u^I_{\alpha})^2(T^{\alpha}_{ij})^2 \\
&-\text{Hess}\,G(u_i,u_i)-\text{Hess}\,G(u_{\alpha},u_{\alpha})\\
=&-C_{\epsilon}e_H(u)-\epsilon e_V(u)+\frac 1  2\sum\limits_I\sum\limits_{\alpha}(u^I_{\alpha})^2\eta(e_{\alpha}) \\
&-\text{Hess}\,G(u_i,u_i)-\text{Hess}\,G(u_{\alpha},u_{\alpha})\\
\ge&-(C_{\epsilon}+2\lambda_G)e_H(u)+(\eta_{min}-2\lambda_G-\epsilon)e_V(u).
\end{split}
\end{equation}
Integrating (\ref{eq:step-2 bochner}) over $M$ shows
\begin{equation*}
\frac {\mathrm{d}}{\mathrm{d}t}E(u) \le (C_{\epsilon}+2\lambda_G)E_H(u)-(\eta_{min}-2\lambda_G-\epsilon) E_V(u).
\end{equation*}
Consequently
\begin{equation*}
\begin{split}
\frac {\mathrm{d}}{\mathrm{d}t}E_V(u)+\frac{(\eta_{min}-2\lambda_G)}{2} E_V(u)
\le &\frac {\mathrm{d}}{\mathrm{d}t}E_P(u)-\frac {\mathrm{d}}{\mathrm{d}t}E_G(u)\\
&+(C_{\epsilon}+\eta_{min})(E_G(u(\cdot,t_0))+\vert E_P\vert).
\end{split}
\end{equation*}
From (\ref{eq:heat flow u}) and (\ref{eq:first derivative u}), we get
\begin{equation*}
\frac{\mathrm{d}}{\mathrm{d}t}E_G(u(\cdot,t))=-\int_M\vert \frac{\partial (u(\cdot,t))}{\partial t} \vert^2 \mathrm{d}v_g \le 0.
\end{equation*}
Clearly, we have (\ref{eq:time derivative subsolution}) with $C_1=2\lambda_G$.
Integrating (\ref{eq:time derivative subsolution}) over $M\times(t_0,t)$, we get
\begin{equation}\label{eq:time derivative intergral estimate}
\int_M \vert\frac{\partial (u(\cdot,t))}{\partial t}\vert^2\mathrm{d}v_g\le\int_M \vert\frac{\partial (u(\cdot,t_0))}{\partial t}\vert^2\mathrm{d}v_g+\eta_{min} \int_{t_0}^{t}\int_M \vert\frac{\partial (u(\cdot,s))}{\partial t}\vert^2\mathrm{d}v_g \mathrm{d}s
\end{equation}
since $2\lambda_G<\eta_{min}$.
On the other hand, by integrating (\ref{eq:first derivative u}) on $(t_0,t)$, we obtain
\begin{equation}\label{eq:energy control}
E_G(u(\cdot,t_0))+\vert E_P\vert \ge E_G(u(\cdot,t_0))-E_G(u(\cdot,t))= \int_{t_0}^{t}\int_M \vert\frac{\partial (u(\cdot,s))}{\partial t}\vert^2\mathrm{d}v_g \mathrm{d}s.
\end{equation}
Then we have
\begin{equation*}
\begin{split}
\frac{\mathrm{d}}{\mathrm{d}t}E_G(u(\cdot,t))&\ge\frac{\mathrm{d}}{\mathrm{d}t}E_G(u(\cdot,t_0))-\eta_{min} \Big(E_G(u(\cdot,t_0))+\vert E_P\vert\Big)\\
&=-\int_M\vert\tau(u(\cdot,t_0))\vert^2 \mathrm{d}v_g-\eta_{min} \Big(E_G(u(\cdot,t_0))+\vert E_P\vert\Big).
\end{split}
\end{equation*}
Notice that
\begin{equation*}
\begin{split}
\vert\frac {\mathrm{d}}{\mathrm{d}t}E_P\vert&=\vert\int_M\langle\tau,\nabla G\rangle \mathrm{d}v_g\vert \\
&\le {\frac 1 2}(\int_M\vert\tau(u(\cdot,t))\vert^2 \mathrm{d}v_g + \int_M\vert\nabla G\vert^2 \mathrm{d}v_g) \\
&\le {\frac 1 2}(\int_M\vert\tau(u(\cdot,t_0))\vert^2 \mathrm{d}v_g+\eta_{min} (E_G(u(\cdot,t_0))+\vert E_P\vert) + \int_M\vert\nabla G\vert^2 \mathrm{d}v_g)\\
&<\infty
\end{split}
\end{equation*}
and
\begin{equation*}
\vert E_P\vert<(\text{max}_N\vert G \vert)\text{vol}(M)<\infty,
\end{equation*}
since $M$ is compact.
Set
\begin{equation*}
\begin{split}
A=&2\int_M\vert\tau(u(\cdot,t_0))\vert^2 \mathrm{d}v_g+\int_M\vert\nabla G\vert^2 \mathrm{d}v_g \\
&+(C_{\epsilon}+3\eta_{min})\Big(E_G(u(\cdot,t_0))+(\text{max}_N\vert G \vert)\text{vol}(M)\Big).
\end{split}
\end{equation*}
It follows that
\begin{equation*}
\frac {\mathrm{d}}{\mathrm{d}t}E_V(u(\cdot,t))+\frac{\eta_{min}-2\lambda_G}{2} E_V(u(\cdot,t))\le A,
\end{equation*}
that is,
\begin{equation}\label{eq:vertical energy}
\frac {\mathrm{d}}{\mathrm{d}t}\Big(e^{\frac{\eta_{min}-2\lambda_G}{2} t}E_V(u(\cdot,t))\Big)\le Ae^{\frac{\eta_{min}-2\lambda_G}{2} t}.
\end{equation}
By integrating (\ref{eq:vertical energy}) over $[t_0,t]$, we get
\begin{equation*}
\Big(e^{\frac{\eta_{min}-2\lambda_G}{2}  s}E_V(u(\cdot,s))\Big)\vert_{s=t_0}^{s=t}\le \frac{2A}{\eta_{min}-2\lambda_G}(e^{\frac{\eta_{min}-2\lambda_G}{2}  s})\vert_{s=t_0}^{s=t}.
\end{equation*}
Therefore
\begin{equation*}
\begin{split}
E_V(u(\cdot,t))&\le e^{\frac{\eta_{min}-2\lambda_G}{2} (t_0-t)}E_V(u(\cdot,t_0))\\
&\quad+\frac{2A}{\eta_{min}-2\lambda_G}(1-e^{\frac{\eta_{min}-2\lambda_G}{2} (t_0-t)}) \\
&\le E_V(u(\cdot,t_0))+\frac{2A}{\eta_{min}-2\lambda_G}.
\end{split}
\end{equation*}
\end{proof}

\begin{proof}[Proof of Theorem \ref{thm:eells sampson type for compact}]
From Theorem \ref{thm:long time existence for complete}, we know that one can solve (\ref{eq:heat flow u}) for all time.
From (\ref{eq:energy control}), we get
\begin{equation*}
\begin{split}
\int_0^{\delta}\int_M\vert \frac{\partial u}{\partial t}(s)\vert^2\,\mathrm{d}v_g\,\mathrm{d}s&=E_G(\bar u)-E_G(u(\delta))\\
&\le E_G(\bar u)+\vert E_P(u(\delta))\vert \\
&\le E_G(\bar u)+(\text{max}_N\vert G \vert)\text{vol}(M)
\end{split}
\end{equation*}
which leads to
\begin{equation*}
\int_0^{\infty}\int_M\vert \frac{\partial u}{\partial t}(s)\vert^2\,\mathrm{d}v_g\,\mathrm{d}s<\infty.
\end{equation*}
Hence there exists a sequence $s_n\to\infty$ such that
\begin{equation}\label{eq:L2 vanishing sequence}
\int_M\vert \frac{\partial u}{\partial t}(s_n)\vert^2 \mathrm{d}v_g\to 0.
\end{equation}
Due to (\ref{eq:time derivative subsolution}), we know
\begin{equation*}
(\Delta_H-\frac{\partial}{\partial t})\exp(-Ct)\vert\frac{\partial u}{\partial t}(t)\vert^2\ge 0.
\end{equation*}
The function
\begin{equation*}
\phi(x,t)=\exp\big(-C(s+t)\big)\vert\frac{\partial u}{\partial t}(s+t)\vert^2
\end{equation*}
also satisfies
\begin{equation*}
(\Delta_H-\frac{\partial}{\partial t})\phi\ge 0.
\end{equation*}
By Lemma \ref{thm:subelliptic maximum principle}, we obtain
\begin{equation*}
\vert \frac{\partial u}{\partial t}(s+t)\vert^2\le B\exp(Ct)t^{-\frac Q 2}\int_M\vert \frac{\partial u}{\partial t}(s)\vert dv_g
\end{equation*}
for $0<t<R_0^2$.
Then, for $t=\frac{R^2_0} 2$, we have
\begin{equation}\label{eq:ft4}
\vert \frac{\partial u}{\partial t}(s+\frac{R^2_0} 2)\vert^2\le\frac{2^{\frac Q 2}B}{R^Q_0}\exp(\frac{CR^2_0} 2)\int_M\vert\frac{\partial u}{\partial t}(s)\vert^2dv_g
\end{equation}
for any $s>0$. From (\ref{eq:L2 vanishing sequence}) and (\ref{eq:ft4}), it follows that
\begin{equation}\label{eq:vanishing sequence}
\underset{x\in M}{\text{sup}}\vert \frac{\partial u}{\partial t}(s_n+\frac{R^2_0} 2)\vert^2\to 0
\end{equation}
as $n\to\infty$.
By Lemma \ref{thm:energy density} and Lemma \ref{thm:step-2}, we get
\begin{equation*}
e(u(x,t))\in L^{\infty}(M\times[0,+\infty)).
\end{equation*}
Setting $t_n=s_n+\frac{R^2_0} 2$, owing to the compactness of $N$ and the uniform boundedness of $e(u(\cdot,t_n))$, the sequence $\{u(\cdot,t_n)\}$ form a uniformly bounded and equicontinuous family of maps. Hence, according to Arzela-Ascoli Theorem, there exists a subsequence $t_{n_k}\to\infty$ such that
\begin{equation}\label{eq:convergent subsequence}
u(\cdot,t_{n_k})\to u_{\infty}(\cdot)
\end{equation}
to a Lipschitz map $u_{\infty}:M\to N\subset \mathbb{R}^K$.

From (\ref{eq:vanishing sequence}) and (\ref{eq:convergent subsequence}), we know that $u_{\infty}$ solves (\ref{eq:euler-lagrange}) weakly. By Theorem \ref{thm:subelliptic regularity theory}, $u_{\infty}$ is smooth. Since $u(t,x)$ is smooth in $t$, then $u_{\infty}$ is homotopic to $u(0) = \bar u$.
\end{proof}

When $(N,h)$ is a complete non-compact Riemannian manifold, the solutions $u(\cdot,t)$ of (\ref{eq:heat flow u}) may not be uniformly bounded with respect to $t\in [0,\infty)$. However, if we add a decay condition on the potential function $G$ and also the non-positive curvature assumption on the target manifold $N$, the solution $u(\cdot,t)$ will remain uniformly bounded. By a similar argument for Theorem \ref{thm:eells sampson type for compact}, we can establish the Eells-Sampson type theorems too.

\begin{proof}[Proof of Proposition \ref{thm:hess assumption}]
The global existence of the solution $u(t)$ is given by Theorem \ref{thm:long time existence for complete}. It is enough to show that there
exists a fixed compact set $N' \subset N$ such that $u(t)(M) \subset N'$ for all $t \in [0, +\infty)$.

To this end, we set
\begin{equation*}
f=\vert\frac{\partial u(t)}{\partial t}\vert^2 \qquad \text{and} \qquad \phi(t)=\underset {x\in M}{\text{sup}}\,\sqrt {f(t,x)}.
\end{equation*}
By the completeness of $N$, for each $x\in M, t\in [0, +\infty)$, there exists a minimal geodesic $\gamma_x$ connecting $\bar u(x)$ and $u(x,t)$, whose length is $d_N(\bar u(x),u(x,t))$. Then we have the following triangle inequality
\begin{equation}\label{eq:triangle inequality}
\rho(u(x,t))\le \rho(\bar u(x))+ d_N(\bar u(x),u(x,t))
\end{equation}
where $\rho$ denotes the distance function on $N$ from the fixed point $P_0\in N$.
Note that
\begin{equation}\label{eq:length control}
d_N(\bar u(x),u(x,t))\le \int_0^t \phi(s)\,\mathrm{d}s.
\end{equation}
From (\ref{eq:triangle inequality}) and (\ref{eq:length control}), we get
\begin{equation}\label{eq:distance control}
\rho (u(x,t))\le C_1+\int_0^t \phi(s)\,\mathrm{d}s
\end{equation}
where $C_1=\underset{x\in M}{\text{max}}\,\rho(\bar u(x))$.
Since $f$ satisfies
\begin{equation}\label{eq:control by hess}
\begin{split}
(\frac{\partial}{\partial t}-\Delta_H)f
&=-2\vert\nabla^{H} \frac{\partial u(t)}{\partial t}\vert^2+2\text{Hess}\,G(\frac{\partial u(t)}{\partial t},\frac{\partial u(t)}{\partial t})\\
&+2\langle\text{Riem}_{N}(du(t)(e_i),\frac{\partial u(t)}{\partial t})du(t)(e_i),\frac{\partial u(t)}{\partial t}\rangle \\
&\le 2\text{Hess}\,G(\frac{\partial u(t)}{\partial t},\frac{\partial u(t)}{\partial t}),
\end{split}
\end{equation}
from (\ref{eq:hess assumption}) and (\ref{eq:distance control}), the inequality (\ref{eq:control by hess}) becomes
\begin{equation*}
(\frac{\partial}{\partial t}-\Delta_H)f\le -C(1+\int_0^t \phi(s)\,\mathrm{d}s)^{-1}f.
\end{equation*}
Next, setting $g=\exp(\psi)f$ with $\psi(t)=C\int_0^t(1+\int_0^{\tau} \phi(s)\mathrm\, {d}s)^{-1}\,\mathrm{d}\tau$, we get
\begin{equation*}
(\frac{\partial}{\partial t}-\Delta_H)g\le 0
\end{equation*}
and it follows from Lemma \ref{thm:subelliptic maximum principle} that
\begin{equation*}
\underset {x\in M}{\text{sup}}\,g(t,\cdot)\le \underset {x\in M}{\text{sup}}\,g(0,\cdot),
\end{equation*}
that is
\begin{equation}\label{eq:point estimate}
\begin{split}
\phi(t)\le &\phi(0)\exp(\frac{-\psi(t)}{2})\\
=&\phi(0)\exp[-\frac{C}{2}\int_0^t(1+\int_0^{\tau}\phi(s)\,\mathrm{d}s)^{-1}\,\mathrm{d}\tau].
\end{split}
\end{equation}
Since $1+\int_0^{\tau} \phi(s)\,\mathrm{d}s\le 1+\phi(0)\tau$, then $\psi(t)\ge \frac{C\ln (1+\phi(0)t)}{\phi(0)}$. Substituting this into (\ref{eq:point estimate}), we have
\begin{equation*}
\phi(t)\le \phi(0)\exp(\frac{-C\ln (1+\phi(0)t)}{2\phi(0)})=\frac{\phi(0)}{(1+\phi(0)t)^{\frac{C}{2\phi(0)}}}
\end{equation*}
which suggests that $\phi(t)\to 0$, as $t\to +\infty$. Then, for any $C_2>0$, there exists $t_0>0$ such that
\begin{equation*}
\int_0^{t}\phi(s)\,\mathrm{d}s \le C_2 t
\end{equation*}
for all $t\ge t_0$.
Hence we have
\begin{equation}\label{eq:phi inequality}
\begin{split}
\phi(t)\le &\phi(0)\exp[-\frac{C}{2}\int_0^t(1+\int_0^{\tau}\phi(s)\,\mathrm{d}s)^{-1}\,\mathrm{d}\tau]\\
\le &\phi(0)\exp[-\frac{C}{2}\int_{t_0}^t(1+\int_0^{\tau}\phi(s)\,\mathrm{d}s)^{-1}\,\mathrm{d}\tau]\\
\le &\phi(0)\exp[-\frac{C}{2}\int_{t_0}^t(1+C_2\tau)^{-1}\,\mathrm{d}\tau]\\
= &\phi(0)\exp[-\frac{C}{2C_2}(\ln(1+C_2t)-\ln(1+C_2t_0))]\\
= & \frac{\phi(0)(1+C_2t_0)^{\frac{C}{2C_2}}}{(1+C_2t)^{\frac{C}{2C_2}}}
\end{split}
\end{equation}
for $t\ge t_0$.
Choosing a sufficiently small $C_2$ such that $\frac{C}{2C_2}>1$, integrating (\ref{eq:phi inequality}) over $[t_0,t]$ then gives
\begin{equation*}
\int_{t_0}^{t}\phi(s)\,\mathrm{d}s \le C_3 \quad \text{for all} \quad t\ge t_0
\end{equation*}
which leads to
\begin{equation}\label{eq:crude phi}
\int_{0}^{+\infty}\phi(s)\,\mathrm{d}s \le C_4
\end{equation}
where $C_3$, $C_4$ are positive constants.
Using (\ref{eq:crude phi}) in (\ref{eq:point estimate}), we get for some positive constants $C_5$ and $C_6$
\begin{equation}\label{eq:phi}
\phi(t)\le C_5 e^{-C_6t}.
\end{equation}
In terms of (\ref{eq:distance control}) and (\ref{eq:phi}), we get $\rho(u(t)) \le C_0$. So, there exists a compact set $N'\subset N$ such that, for all $t \in [0, +\infty)$, $u(t)(M) \subset N'$. It follows that there exists a sequence $t_k \to +\infty$ such that $u_{t_k}$ converges to $u_{\infty}$ which is a subelliptic harmonic map with potential $G$. To see that $u(t) \to u_{\infty}$, as $t \to +\infty $, we note that
$d_N (u(t),u_{\infty})\le d_N (u(t),u({t_k}) )+d_N (u({t_k}) ,u_{\infty})$.
Thanks to (\ref{eq:phi}), we obtain
\begin{equation*}
d_N(u(t),u(t_k))\le \int_{t_k}^t \vert\frac{\partial u(s)}{\partial s}\vert\,\mathrm{d}s\le C_5\int_{t_k}^t e^{-C_0t}\,\mathrm{d}s \to 0 \quad as\quad k,t\to \infty.
\end{equation*}
Finally, from the following formula
\begin{equation*}
\begin{split}
\Delta_H(-G \circ u)&=-\mathrm{d}G(\tau_H(u))- \text{Trace}_g\text{Hess}\, G(\mathrm{d}u_H, \mathrm{d}u_H)\\
&= \vert\widetilde\nabla G(u)\vert^2 - \text{Trace}_g\text{Hess}\, G(\mathrm{d}u_H, \mathrm{d}u_H) \ge 0,
\end{split}
\end{equation*}
we see that $(-G\circ u)$ is a subsolution of subelliptic harmonic equation and so is constant. Since $G$ is strictly concave, we have
$\text{Trace}_g\text{Hess}\, G(\mathrm{d}u_H, \mathrm{d}u_H)=0$, so $u$ is constant.
\end{proof}

The Riemannian foliation $(M, g; \mathfrak{F})$ will be said to be tense if the mean vector field of $\mathfrak{F}$ is parallel with respect to the Bott connection along the leaves. For a compact sub-Riemannian manifold $(M^{m+d}, H, g_{H}; g)$ corresponding to a tense Riemannian foliation $(M, g; \mathfrak F)$, we have the following lemma.

\begin{lemma}\label{thm:riemannian foliation}
Let $(M^{m+d}, H, g_{H}; g)$ be a compact sub-Riemannian manifold corresponding to
a tense Riemannian foliation $(M, g; \mathfrak F)$. Let N be a complete Riemannian manifold with non-positive sectional curvature. Suppose $u:M\times[0,\delta)\to N$ is a solution of the subelliptic harmonic map heat flow with potential $G$ and $\text{Hess}\,G\le 0$, then $E_V(u(\cdot,t))$ is decreasing. In particular, $E_V(u(\cdot,t))\le E(\bar u)$.
\end{lemma}

\begin{proof}
The assumption that $(M, g; \mathfrak{F})$ is tense implying that $\nabla^{\mathfrak B}_{\xi} \zeta= 0$, for any $\xi \in V$ and the curvature tensor of $\nabla^{\mathfrak B}$ satisfies [cf. \cite{MR4236537}]
\begin{equation*}
R^A_{j\alpha k}=0.
\end{equation*}
In particular, we have $R^j_{k\alpha k}=0$. From (4.29) in \cite{MR4236537}, we have
\begin{equation*}
\begin{split}
\Delta_H e_V(u_t)=&(u^I_{\alpha k})^2+u^I_{\alpha}\tau^I_{H,\alpha}+u^I_{\alpha}\zeta^k_{,\alpha}u^I_k \\
&+u^I_{\alpha}u^I_j R^j_{k\alpha k}-u^I_{\alpha}u^K_k\widetilde R^I_{KJL}u^J_{\alpha}u^L_k.
\end{split}
\end{equation*}
Consequently,
\begin{equation}\label{eq:vertical bochner}
\begin{split}
(\Delta_H-\frac{\partial}{\partial t})e_V(u_t)=&(u^I_{\alpha k})^2+u^I_{\alpha}\zeta^k_{,\alpha}u^I_{k}+u^I_{\alpha}u^I_jR^j_{k\alpha k}\\
&-u^I_{\alpha}u^K_k\widetilde R^I_{KJL}u^J_{\alpha}u^L_k -u^I_{\alpha}G_{IJ}u^J_{\alpha}\\
=&(u^I_{\alpha k})^2-u^I_{\alpha}u^K_k\widetilde R^I_{KJL}u^J_{\alpha}u^L_k -u^I_{\alpha}G_{IJ}u^J_{\alpha}\\
\ge&0.
\end{split}
\end{equation}
Integrating (\ref{eq:vertical bochner}) then gives this lemma.
\end{proof}

\begin{proof}[Proof of Proposition \ref{thm:hess assumption foliation}]
According to Lemma \ref{thm:riemannian foliation}, $E_V(f)$ is uniformly bounded, since $\text{Hess}\,G \le 0$. Using a similar argument for Theorem \ref{thm:eells sampson type for compact} and Proposition \ref{thm:hess assumption}, the proposition follows.
\end{proof}

\begin{proof}[Proof of Proposition \ref{thm:product assumption}]
Let $u(x,t)$ be a solution of (\ref{eq:euclidean heat flow}). Assumption (\ref{eq:product assumption}) is equivalent to
\begin{equation*}
(\frac{\partial}{\partial t}-\Delta_H)\langle u(x,t),u(x,t)\rangle_{\mathbb {R}^{K}}\le 0.
\end{equation*}
By Lemma \ref{thm:subelliptic maximum principle}, we get
\begin{equation*}
\underset {x\in M}{\text{sup}}\,\langle u(x,t),u(x,t)\rangle_{\mathbb {R}^{K}}\le \underset {x\in M}{\text{sup}}\,\langle \bar u(x),\bar u(x)\rangle_{\mathbb {R}^{K}}.
\end{equation*}
Since $\bar u(x)$ is included in a compact set, we have $\underset {x\in M}{\text{sup}}\,\langle \bar u(x),\bar u(x)\rangle_{\mathbb {R}^{K}}\le C$. It follows that $\underset {x\in M}{\text{sup}}\,\langle u(x,t),u(x,t)\rangle_{\mathbb {R}^{K}}\le C$, which implies $u(t)(M)$ is included in a compact set. By Theorem \ref{thm:long time existence for complete}, we have $T = +\infty$. In both cases, since $E_V(f)$ is uniformly bounded, the convergence follows from Theorem \ref{thm:eells sampson type for compact} immediately. When $(M,H,g_H;g)$ is a compact sub-Riemannian manifold corresponding to
a tense Riemannian foliation, by argument of Proposition \ref{thm:hess assumption}, we know $u_{\infty}$ is constant.
\end{proof}

\end{document}